\documentclass{article}[12pt]

\usepackage{amssymb, amsmath, amsthm, graphicx}

      \theoremstyle{plain}
      \newtheorem{theorem}{Theorem}[section]
      \newtheorem{lemma}[theorem]{Lemma}
      \newtheorem{corollary}[theorem]{Corollary}
      \newtheorem{proposition}[theorem]{Proposition}
      
      \theoremstyle{definition}
      \newtheorem{definition}[theorem]{Definition}
      
      \theoremstyle{remark}
      \newtheorem{remark}[theorem]{Remark}

\begin{document}

\title{A non-Markovian model of rill erosion}

\author{Michael Damron\thanks{Courant Institute of Mathematical Sciences, 251 Mercer St., New York, NY 10012, USA. Email: damron@cims.nyu.edu; Research supported in part by NSF grant number OISE-0730136.}
\and
C.L. Winter\thanks{National Center for Atmospheric Research, 1850 Table Mesa Dr., Boulder, CO 80305, USA.  Email: lwinter@ucar.edu}
}
\date{December 2008}
\maketitle

\footnotetext{** See arXiv:0704.2706 and arXiv:math/0702542}

\begin{abstract}
We introduce a new model for rill erosion.
We start with a network similar to that in the Dynamical Discrete Web** and
instantiate a dynamics which makes the process highly non-Markovian.
The behavior of nodes in the streams is similar to the behavior of
Polya urns with time-dependent input.  In this paper we use a combination of rigorous arguments and simulation results to show that the model exhibits many properties of rill erosion; in particular, nodes which are deeper in the network tend to switch less quickly.
\end{abstract}


\section{Introduction}
\subsection{Reinforcement and Rill Erosion}

Stochastic processes with reinforcement are inherently non-Markovian and therefore may model some real phenomena more accurately than can their Markovian counterparts.  Reinforcement is a mechanism that provides a bias to a system, making it more likely to occupy states the more often those states are visited. Some well-studied examples include variations on the urn of P\'{o}lya (the original introduced in \cite{polya} and this and subsequent models studied, for example, in \cite{athreya} and \cite{friedman}) and reinforced random walks (\cite{diaconis, pemantle2}). The infinite memory exhibited in these examples can force a system to spend most (or almost all) of its time in a small subset of its state space.  Many natural phenomena exhibit similar behavior; for instance, the overall pattern of erosion on a hillslope is relatively stable once it is established, although small details of the pattern may change frequently and catastrophes that permanently alter it may occasionally occur.

We investigate a discrete time, infinite-memory random process defined on the nodes and edges of an oriented diagonal lattice (Figure \ref{nodenetworkfig}) that we propose as a simple model of hillslope erosion.  The lattice starts out smooth in the sense that it has no edges initially, but it sprouts edges everywhere the instant the process starts, much as rain can start soil erosion everywhere on a hillslope at once.  Edges may connect an interior node to two, one, or neither of the two nodes directly above it.  Exactly one edge descends from each interior node, and it points either left or right.  At every node and at every time step a simple two parameter reinforcing law, based on the entire history of the network above a given interior node, randomly determines the direction of the nodeÕs descending edge and then is updated.  Obvious modifications of these statements apply to nodes at the top or bottom (if one exists) of the lattice. 

The current pattern of connections among nodes represents the present state of the process, and the patternÕs stability -- measured by the tendency of the same state, or one similar to it, to occur on subsequent iterations of the process -- represents the patternÕs strength as a memory.  The degree of reinforcement is set by tuning two parameters, $r$ and $\alpha$.  At any given moment the current pattern is a collection of dendritic networks that appears similar to drainage networks found in nature; indeed, lattice models have often been used to investigate the morphology of natural drainage networks (e.g. \cite{RRI}).  We focus on the surficial dynamics of rill networks \cite{glock}, rather than their morphology.  Put in terms of erosion, we are more interested in the process of erosion than we are in the result. 

The analogy between our model and erosion, specifically rill erosion, is straightforward: $r$ can be interpreted as a rainfall rate (or equivalently, as the rate of sediment generation) and $\alpha^{-1}$ as the resistance of soil to erosion, while the reinforcement dynamics correspond to the overland flow of water and sediment down a hill.  Rills are small, ephemeral channels that transport sediment down hillslopes when it rains \cite{toy}. They form when rainfall and runoff dislodge particles from the soil surface and transport them along flow paths governed by variations in the surface roughness of soils and the soil's ability to resist erosion.  Flow depths in rills are typically on the order of a few centimeters or less, while the longest channels in rill networks can be several meters long.  Processes affecting rill erosion take place over timescales ranging from milliseconds to hours.

The topology of rill networks is relatively unstable when compared to larger scale natural drainage systems (of which rills may be a part) like gulley systems and river basins.  Rill networks are most unstable at their tops where boundaries between rills and inter-rill areas are not well defined and shift often, but connectivity can change downhill as well, usually at a slower rate than uphill.  Some rills grow throughout a rainfall event, others are filled by sediment and disappear, still others alternate.  A detailed description of rill erosion 1) must account for complicated interactions among rainfall, soil properties, and topography \cite{WBR}, and 2) often depends on obtaining a set of physical parameters that are difficult to measure.

Despite the high degree of complexity of rill erosion at small scales, at macroscopic scales it is principally determined by particle detachment, overland flow, and sediment transport (\cite{quinton}).  In turn, each of flow, detachment, and transport depends critically on the rate of rainfall and the soilÕs resistance to erosion.  It is not completely surprising that our simple two parameter model exhibits some important elements of the macroscopic behavior of rills formation.  In fact, similar to rill erosion, each node in the model network switches direction infinitely many times but the switching rate depends on position up or down hill.  Furthermore, floods that carry unusually large amounts of water and catastrophes that significantly alter the flow pattern occur occasionally in the model, as they do in nature.

\subsection{Definition of the Model}

Consider the vertices in the even sub-lattice of ${\mathbb Z}^2$ which have second coordinate non-positive.  That is, the set

\[ {\mathbb Z}_{even}^2 = \{(x,y) \in {\mathbb Z}^2 : x+y \textrm{ is even and } y \leq 0 \} \] 
and edges

\[ {\mathbb E}_{even}^2 = \{ <(x,y),(x+1,y-1)> : (x,y) \in {\mathbb Z}_{even}^2 \} \]

\[ \cup \{ <(x,y),(x-1, y-1)> : (x,y) \in {\mathbb Z}_{even}^2 \} \]
Let $v = (x,y)$ be a node with left parent $w_1=(x-1, y+1)$, right parent $w_2=(x+1, y+1)$ (for those nodes with second coordinate 0, parents will not exist), and with left child $(x-1,y-1)$, right child $(x+1,y-1)$.  We will use the term \textit{depth k} to refer to those nodes with $y$ coordinate equal to $1-k$.  Conversely, for any node $v$ the term $\textrm{depth}(v)$ will denote the numerical value of the depth of $v$.

First we describe the algorithm for the behavior of $v$ heuristically.  At the end of the $0^{th}$ second, $v$ receives $I_v(0) = r$ units of rain (but does nothing else).  During the $n^{th}$ second, for $n \geq 1$, the following sequence occurs:

\begin{enumerate}

\item $v$ flips a coin, heads-biased with probability $P_v^L(n)$, which reflects $T_v^L(n)$, the total "sediment" load $v$ has sent to its left child by time $n$.

\item If this coin shows heads (tails), $v$ sends its current input of sediment $I_v(n-1)$ to its left (right) child.  $v$ adds this number to the total sediment, $T_v^L(n)$ ($T_v^R(n)$), it has sent to the left (right) for all time.  

\item $v$ receives sediment load from 0,1, or 2 parents and receives $r$ units of rain.  Call the sum of these two $I_v(n)$.  Increment time and return to step 1.

\end{enumerate}

The evolution of the node's behavior depends on two parameters: the rainfall rate $r > 0$ and a term $\alpha^{-1} > 0$ that resists change.  To make this precise, we make several definitions.  We start by initializing variables.  For each $v \in {\mathbb Z}_{even}^2$, let 

\[ T_v^L(0) = T_v^R(0) = T_v(0) = 0 \]
and

\[ I_v(0) = r \textrm{ , } P_v^L(0) = P_v^R(0) = 1/2 \]
For each $n \geq 1$, $v \in {\mathbb Z}_{even}^2$, we define a Bernoulli variable, $ D_v^L(n)$, the biased coin, that is conditionally independent from vertex to vertex given the variables $\{D_v^L(i): v \in {\mathbb Z}_{even}^2, i < n \}$, with parameter $P_v^L(n-1)$.  Next, let

\[ T_v^L(n) = I_v(n-1)D_v^L(n) + T_v^L(n-1) \]

\[ T_v(n) = I_v(n-1) +T_v(n-1) \textrm{ , } T_v^R(n) = T_v(n) - T_v^L(n) \]
We create the bias for the next coin:

\begin{equation}
\label{leftprob}
P_v^L(n) = \frac{T_v^L(n) + \alpha}{T_v(n) + 2\alpha} = \frac{\frac{T_v^L(n)}{r} + \eta}{\frac{T_v(n)}{r} + 2\eta} 
\end{equation}
where $\eta = \alpha/r$ compares the effect of the rain to the system's inherent resistance to change.  In this paper, we shall always take $r=1$ so that $\eta = \alpha$.  Last we define the input

\[ I_v(n) = 
\begin{cases}
\hspace*{1.5in} r & \textrm{depth}(v) = 1 \\
I_{w_1}(n-1)(1-D_{w_1}^L(n)) + I_{w_2}(n-1)D_{w_2}^L(n) + r& \textrm{otherwise} \\
\end{cases}
\]
and the filtration 

\[ {\cal F}_n = \sigma(\{ D_v^L(k) : v \in {\mathbb Z}_{even}^2, k = 1,...,n \}) \]
See Figure \ref{nodenetworkfig} for an illustration of the process at node $v$.

\begin{figure*}
\begin{center}
\begin{tabular}{ c c }
\scalebox{.21}{\includegraphics*[viewport = 0in 1.4in 10in 9.75in]{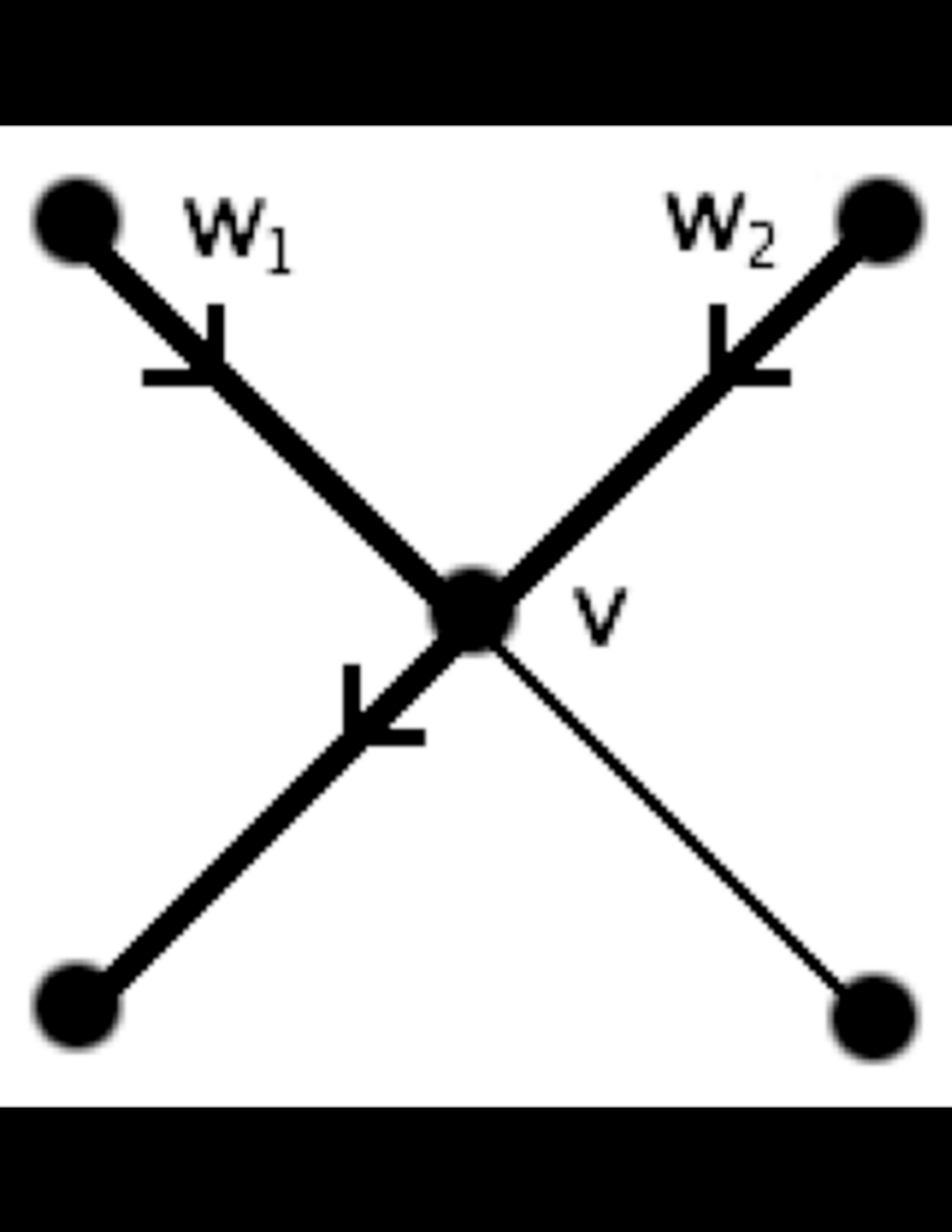}}
\scalebox{.35}{\includegraphics*[viewport = 0in 3in 5.5in 8in]{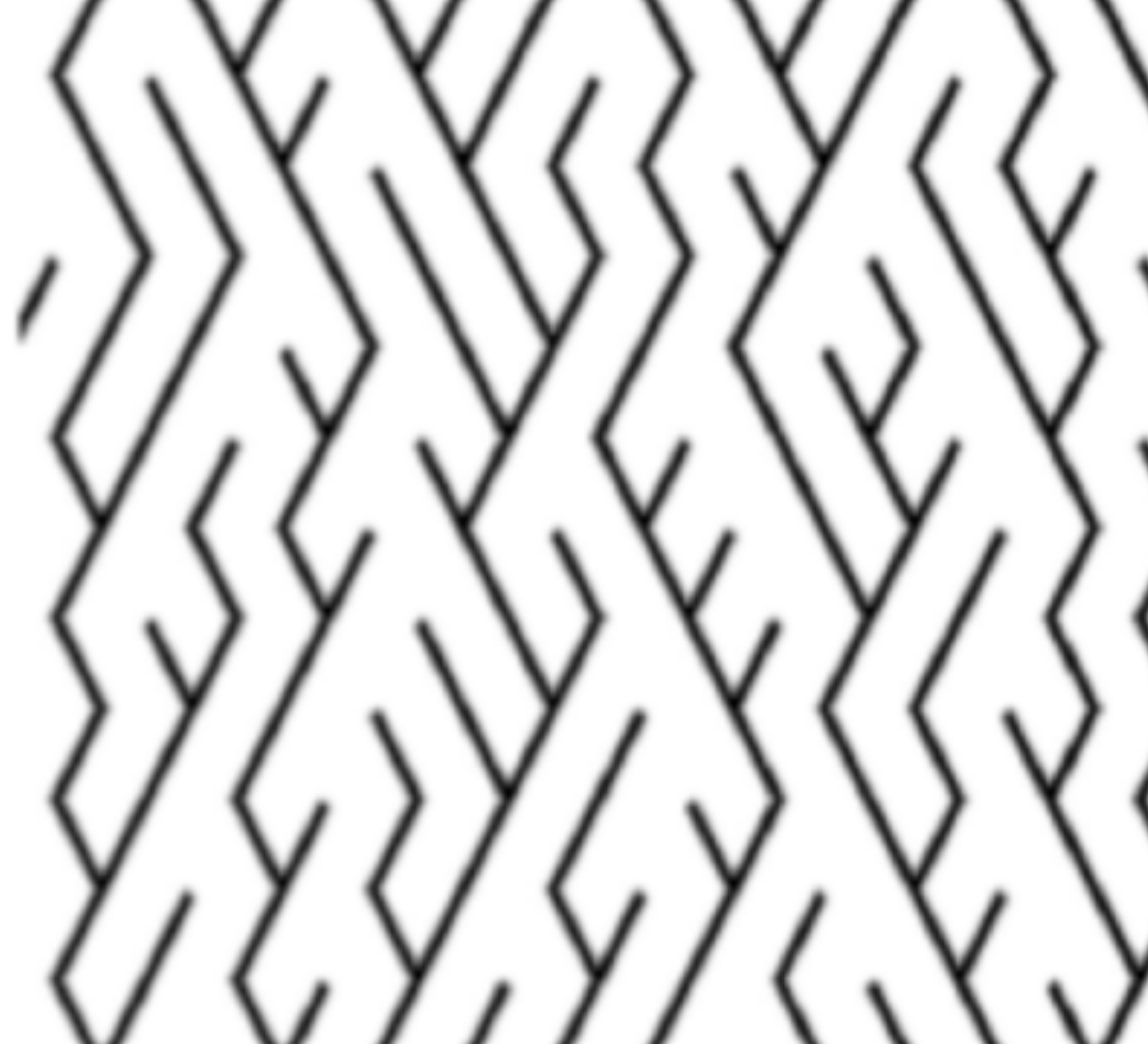}} &
\end{tabular}
\end{center}
\caption{Left: Input and output behavior at node $v$.  The darkened line segments indicate paths of sediment flow.  The light line represents a potential flow path.  Right: A small piece of a visualization of $
\omega_n$ for some $n$.  Only edges $e$ with $\omega_n(e) = 1$ are solid.}
\label{nodenetworkfig}
\end{figure*}

Denote by $d_v = (d_v(1), d_v(2), ...)$ the sequence of directions that node $v$ chooses (for example (L,R,L,...)).  At the end of time $t$, after all nodes have sent loads to their children, we may update certain edge variables.  Define a sequence of edge configurations $\{ \omega_n \}_{n \geq 0}$, where for each $n$, $\omega_n$ is a map from ${\mathbb E}_{even}^2 \to \{0,1\}$, using the following rule.  If the node $v=(x,y)$ has $d_v(n) = L$ then let 

\[ \omega_n(<x-1,y-1>) = 1 \textrm{ , } \omega_n(<x+1,y-1>) = 0 \]  
If, on the other hand, $d_v(n) = R$ then let 

\[ \omega_n(<x-1,y-1>) = 0 \textrm{ , } \omega_n(<x+1,y-1>) = 1 \]
See Figure \ref{nodenetworkfig} for a realization of $\omega_n$.

We say that nodes $v$ and $w$ are connected at time $n$ if there exists a path of distinct adjacent edges $e_1, .., e_m$ with $\omega_n(e_i)=1$ for all $i$ so that $e_1$ connects $v$ to one of its children and $e_m$ connects $w$ to one of its parents.  Denote by $C_{v,n}$ the set of vertices which are connected to $v$ at time $n$ and define the backward (uphill) component of $v=(x,y)$ at time $n$ by

\[ C_{v,n}^+ := C_{v,n} \cap \{(x',y'): y' \geq y \} \]
Finally, let $\omega_n^{-1} = \{ e: \omega_n(e) = 1 \}$.

\subsection{Regimes for $\eta$}

The parameter $\eta$ plays an important role in the behavior of the model.  For a fixed node $v$ (at depth $k$) we have that for all $n \geq 1$

\[ \lim_{\eta \to 0} {\mathbb P}(d_v(n) = L | d_v(0) = R) = 0 \]
This indicates that when $\eta$ is small the node $v$ chooses a direction at time 0 and has a high probability of sticking to this direction for most values of $n \geq 1$.  Since this is true for each node $v$, the evolution of $\{\omega_n\}$ is somewhat simple.  In the limit as $\eta \to 0$, each node picks a direction and stays with that direction for all time.  That is, for each $n \geq 1$, and for each finite subset $E \subset {\mathbb E}_{even}^2$,

\begin{equation}
\label{etazeroeq}
\lim_{\eta \to 0} {\mathbb P}(\omega_0^{-1}(1) \cap E = \omega_n^{-1}(1) \cap E) = 1 
\end{equation}
and the dynamics has no effect on the configuration in any finite subset of ${\mathbb Z}_{even}^2$.  The configurations at any moment are the same as those in the discrete web (\cite{arratia}, \cite{FNRS}). 

In the other direction, as $\eta \to \infty$ each node "forgets" its history.  That is, for each node $v$, the conditional probability given ${\cal F}_n$ that it chooses left at time $n+1$ is given in (\ref{leftprob}), and the limit of this quantity is $1/2$.  By symmetry,

\[ {\mathbb P}(d_v(n+1) = L) = 1/2 = {\mathbb P}(d_v(n+1) = R) \]
and so 

\[ \lim_{\eta \to \infty} \left[ {\mathbb P}(d_v(n+1) = L) - {\mathbb P}(d_v(n+1) = L | {\cal F}_n) \right] = 0. \] Therefore the variables in any finite subset of $\{d_v(n): n \geq 0\}$ converge in distribution to i.i.d. Bernoulli(1/2) variables.  Furthermore, the variables in any finite subset of $\{ d_v(n) : n \geq 0, v \in {\mathbb Z}_{even}^2\}$ converge in distribution to i.i.d. Bernoulli(1/2) variables.  Intuitively this holds because distinct nodes only interact with each other through their input and output loads, and both of these are eventually dominated by large $\eta$.  These statements indicate that when $\eta$ is large, the dynamics of our erosion model are similar to those in a network in which each node flips a fair coin at each time $n$, independently from site to site and from time to time, to determine in which direction to send sediment.  Thus the configurations $\{\omega_n\}$ resemble those taken from the dynamical discrete web (\cite{howittwarren}, \cite{FNRS}).

Given the relation both these extreme cases have to the variables $\{ \omega_n \}$, it is natural to view the present model (with $0 < \eta < \infty$) as an interpolation between the discrete web and the dynamical discrete web.  Indeed, for each fixed $n$, the distribution of $\omega_n$ is the same as that in the case of $\eta = 0$ at time $n=0$ or that in the case of $\eta \to \infty$ at any time $n$.  (In both cases, all directions are chosen by independent fair coin flips).  

As we shall see in section 3.2.2, the model with $0 < \eta < \infty$ can be likened to the case $\eta \to \infty$ in the following way.  Each level $k$ is associated to a measure $\theta_k$ (defined in eq. (\ref{definettidef})).  Each node at level $k$ samples (non-independently) from this measure a value $p_v$.  For any sequence $n_1, n_2, ..., n_m$ of times and for any sequence $x_1,...,x_m$ of elements from the set $\{L,R\}$,

\[ \lim_{T \to \infty} {\mathbb P}(d_v(n_1+T) = x_1, ..., d_v(n_m+T) = x_m) = p_v^{N_L} (1-p_v)^{N_R} \]
where $N_L$ ($N_R$) is the number of $i$ for which $d_i = L$.  Because of this fact, we may view the model (for large time) as one in which each node fixes a Bernoulli parameter $p_v$ and flips a $p_v$-biased coin independently each second $n$ (but not independently from site to site) to determine the direction in which to move sediment.

\subsection{Outline of the Paper}

In section 2 we discuss the (relatively simple) behavior of nodes at depth 1.  Since these nodes receive constant input load over time, we can use the well known model of P\'{o}lya's Urn to analyze their output.  In section 3 we discuss the more complicated behavior of nodes at depth at least 2.  Here we make use of results of Pemantle \cite{pemantle} for the time-dependent P\'{o}lya Urn.  We look more closely at properties of the input load, of the output load, and of the dynamics of these lower-depth nodes.

\section{Top Level}

Since our top level nodes are equivalent to the model of P\'{o}lya's urn, we recall basic facts of P\'{o}lya's model  Start with an urn containing $R_0$ red balls and $B_0$ black balls and draw one ball from the urn.  Return this ball to the urn, along with another ball of the same color.  After this round there are $R_1$ red balls in the urn and $B_1$ black balls in the urn, with either $R_0=R_1$ or $B_0=B_1$.  Repeat this process infinitely many times, creating sequences $\{R_n\}_{n \geq 0}$ and $\{B_n\}_{n \geq 0}$ so that for each $n$,

\[ {\mathbb P}(R_{n+1} - R_n = 1|R_{n}, B_{n}) = \frac{R_n}{R_n+B_n} \]

It is well known that the fraction $F_n^R = \frac{R_n}{R_n+B_n}$ has an almost sure limit and that this limit is distributed as $\beta(R_0,B_0)$ (see e.g. \cite{freedman}).

Let $v$ be a node at depth 1.  At the beginning of each second, $v$ receives an amount of sediment equal to 1 and this input load amount does not change with time.  The node sends this load either to the right or left, depending on the bias rule in (\ref{leftprob}).  We are interested in the fraction of total load the node sends left (right) up to time $n$.  To this end, define the load fraction

\[ LF_v^L(n) = \frac{T_v^L(n)}{T_v(n)} \textrm{ , } LF_v^R(n) = \frac{T_v^R(n)}{T_v(n)} \textrm{ , } n \geq 1 \]

\begin{theorem}
\label{rowonethm}
The quantities $LF_v^L(n)$ and $LF_v^R(n)$ have limits as $n \to \infty$.  These limits are random: they are distributed as $\beta (\eta, \eta)$.
\end{theorem}

\begin{proof} We will indicate the proof only for the case $LF_v^L(n)$.  An easy calculation shows that $P_v^L(n)$ is a martingale w.r.t. $\{{\cal F}_n\}$ and, since it is bounded for all $n$, it has an almost sure limit.  Solving for the limiting distribution is similar to solving for the related quantity in the standard P\'{o}lya urn model.  See, for instance, \cite{feller}.  This gives

\[ \lim_{n \to \infty} P_v^L(n) = \lim_{n \to \infty} \frac{T_v^L(n) + \eta}{T_v(n) + 2\eta} = \lim_{n \to \infty} \frac{\frac{T_v^L(n)}{T_v(n)} + \frac{\eta}{T_v(n)}}{1 + \frac{2\eta}{T_v(n)}} \]
\begin{equation}
\label{pequalsl}
= \lim_{n \to \infty} \frac{T_v^L(n)}{T_v(n)} = \lim_{n \to \infty} LF_v^L(n) 
\end{equation}
because $\eta$ is constant w.r.t. $n$ and $T_v(n) \to \infty$.
\end{proof}

Note that the limiting distribution in Theorem \ref{rowonethm} is supported on [0,1] and has no atoms.  This implies that with probability 1, the node $v$ switches states (L,R) infinitely often and that neither of these states is transient.  This is quite unlike the "sticking" associated to the dynamics in the $\eta \to 0$ limit (refer to (\ref{etazeroeq})).  The distribution from the above theorem for different values of $\eta$ is pictured in Figure \ref{row1fig}.  For $0 < \eta < 1$ the limiting load fraction has a bimodal distribution, and for $\eta > 1$ the distribution is unimodal, symmetric about $\frac{1}{2}$.  This means that when $\eta$ is small, each node is likely to have a relatively strong preference for one direction and that when $\eta$ is large, each node is likely to favor L and R somewhat equally.  The case $\eta = 1$ gives a uniform distribution.  Here $v$ is equally likely to have a strong directional preference as it is not to.

\begin{figure*}
\begin{center}
\begin{tabular}{c c c}
$\eta = \frac{1}{2}$ & $\eta = 1$ & $\eta = 2$ \\
\scalebox{0.17}{\includegraphics*[viewport = 0in 2.65in 8.5in 8.35in]{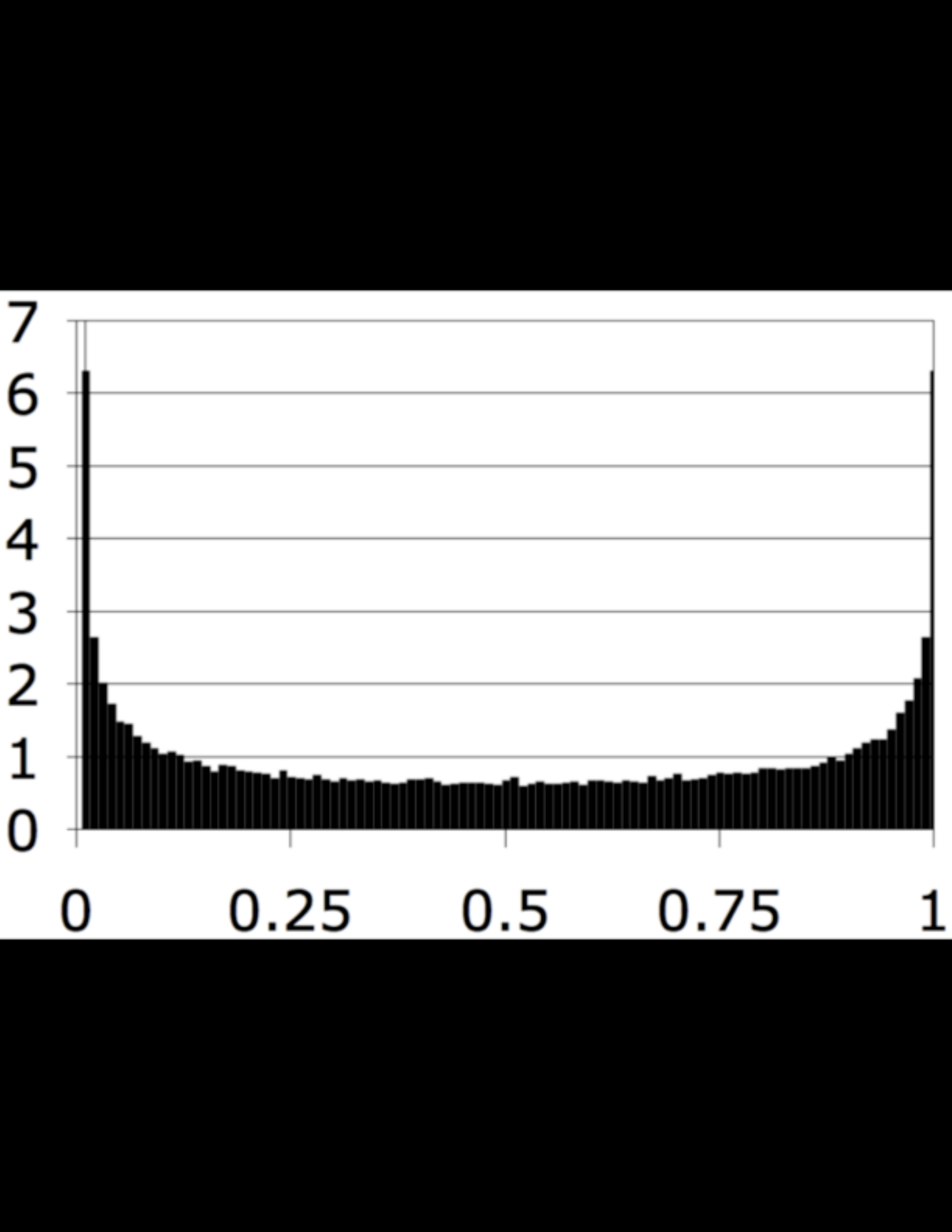}} &
\scalebox{0.17}{\includegraphics*[viewport = 0in 2.65in 8.5in 8.35in]{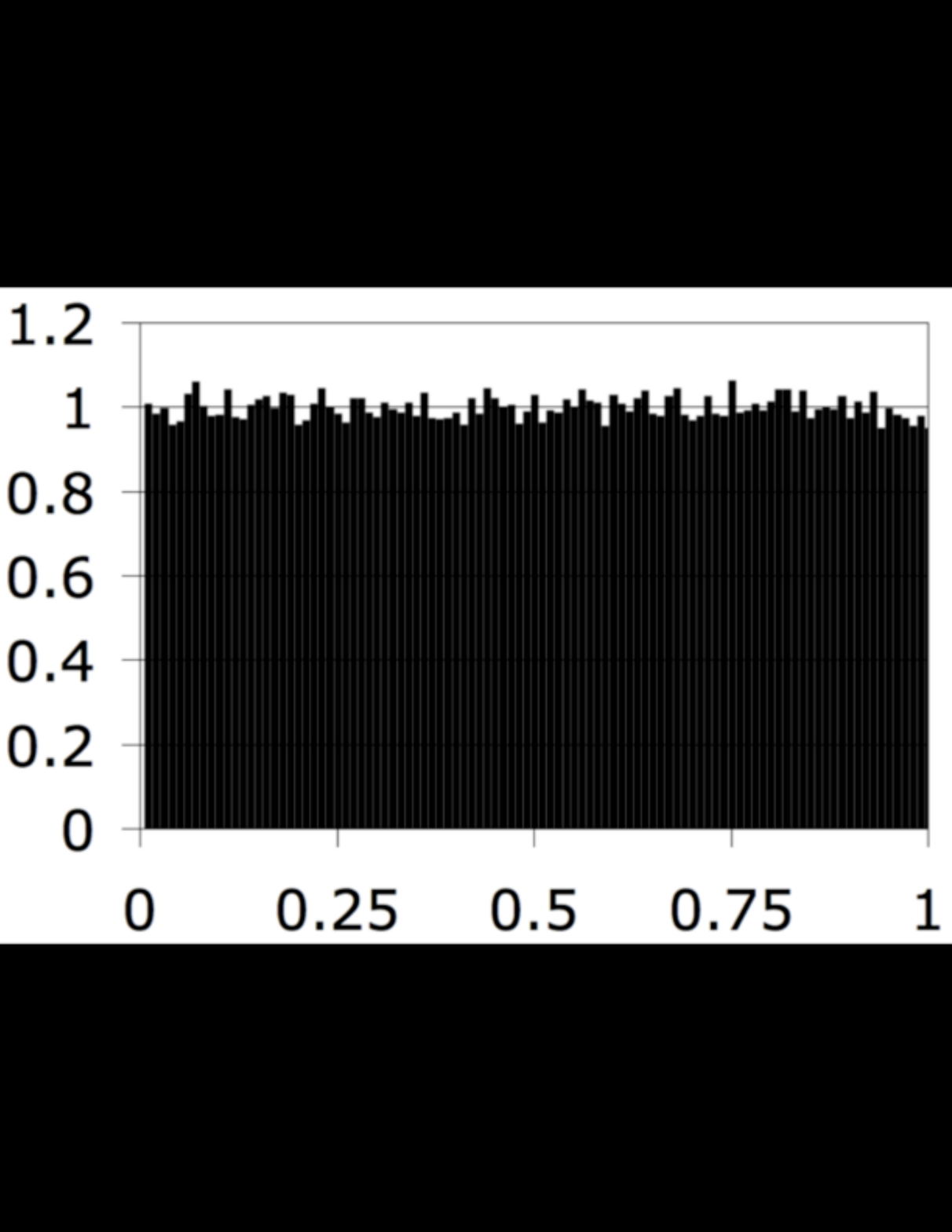}} &
\scalebox{0.17}{\includegraphics*[viewport = 0in 2.65in 8.5in 8.35in]{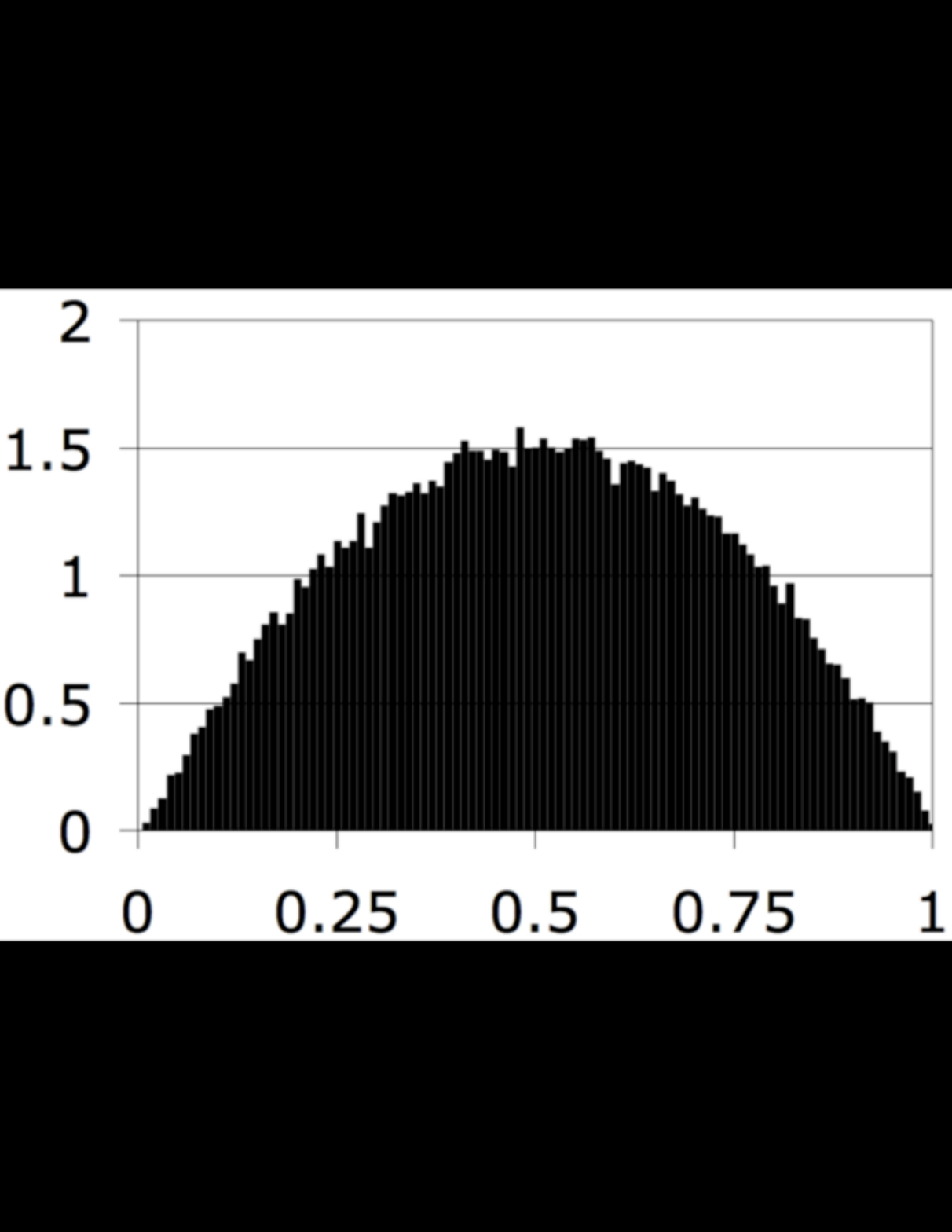}} \\
\end{tabular}
\end{center}
\caption{Asymptotic distributions for $LF_v^L(n)$: $\eta = .5, 1, 2$ respectively.}
\label{row1fig}
\end{figure*}

\section{Lower Levels}

The simplicity of behavior at the top level comes from the fact that each node has an input load which is constant w.r.t. time.  This is not true at lower levels.  Each node has an input load whose magnitude is non-trivially time dependent.  To make this more apparent, isolate an arbitrary node $v$ at depth $k$.  If at time $t=n$, $v$ is not connected to either of its parents in $\omega_n$, then its input load is 1 unit (coming only from rain).  If, on the other hand, $v$ is connected to at least one of its parents, then its input load will be strictly greater than 1 unit.  Therefore, the geometry of the connected components of $\omega_n$ determines the behavior of each node.  This relationship is complex for at least two reasons.  First, not only does the geometry of the network influence node behavior, the node behavior in turn determines the future geometry of the network.  In this sense, our system generates its own randomness.  Second, the method by which this randomness arises involves propagation.  The geometry of nodes at depth $k-l$ at time $m$ affects the behavior of nodes at depth $k$ at time $n$ if and only if $m = n-l$.  In other words, it takes $l$ seconds for the output load from depth $k-l$ to reach nodes at depth $k$.  In spite of these complications, we set out to analyze these lower level nodes.

The node $v$ has an input load sequence $I_v = (I_v(1), I_v(2), ...)$, left output load sequence $T_v^L = (T_v^L(1), T_v^L(2), ...)$, and output direction sequence $d_v = (d_v(1), d_v(2), ...)$.  We are interested in analyzing the nature of this input sequence, the nature of the output sequence, and the relationship between the two.

\subsection{Input Load}
Figure \ref{loadfig} shows a histogram of input load values for all nodes at (a) depth 5, (b) depth 7, and (c) depth 8 at $t= 300$s with $\eta = 1$ (the precise value of $\eta$ does not matter, as a consequence of Theorem \ref{loadthm1}).  The simulation was conducted with periodic boundary conditions, with $10^6$ nodes per row, and with 10 rows.  Therefore, the histogram for depth $k$ at time $n=300$s should closely approximate the probability mass function of the distribution of the input load for depth $k$ at time $n=300$s.  One notices a few things.  First, the support of the distribution at depth $k$ is integers in the interval $[1, \frac{1}{2}k(k+1)]$.  Next, the mass function appears to decrease from load value 1 to a local minimum at $k-1$, to increase for a bit to a local maximum, and then to decrease to the edge of its support.  About $1/4$ of nodes are at the heads of rills, while the fraction of rills starting short of the top increases with depth.  The "bump" in the load distribution to the right of the value $k-1$ appears to travel to the right as depth increases.  Looking at Figure \ref{loadfig}, it is tempting to guess that the load distribution at a given level is a mixture of a distribution for loads that start at the top and one for loads that do not.  Last, the different mass functions have several common values.  For example, the probabilities for load values 1 to 4 are the same in each figure, and the probabilities for load values 1 to 6 are the same in the center and right figures.

\begin{figure*}
\begin{center}
\begin{tabular}{c c c}
Row 5 & Row 7 & Row 8 \\
\scalebox{0.17}{\includegraphics*[viewport = 0in 2.65in 8.5in 8.35in]{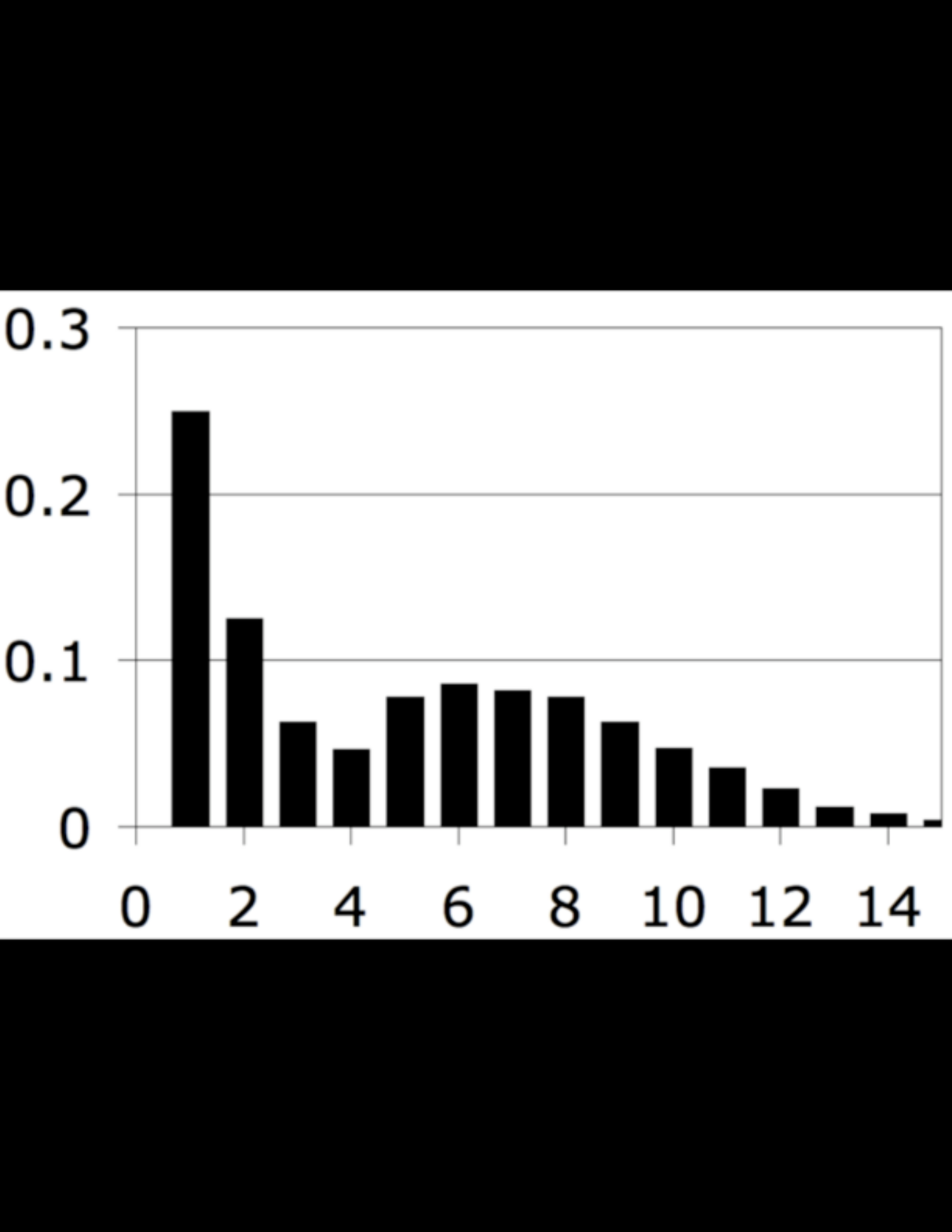}} &
\scalebox{0.17}{\includegraphics*[viewport = 0in 2.72in 8.5in 8.25in]{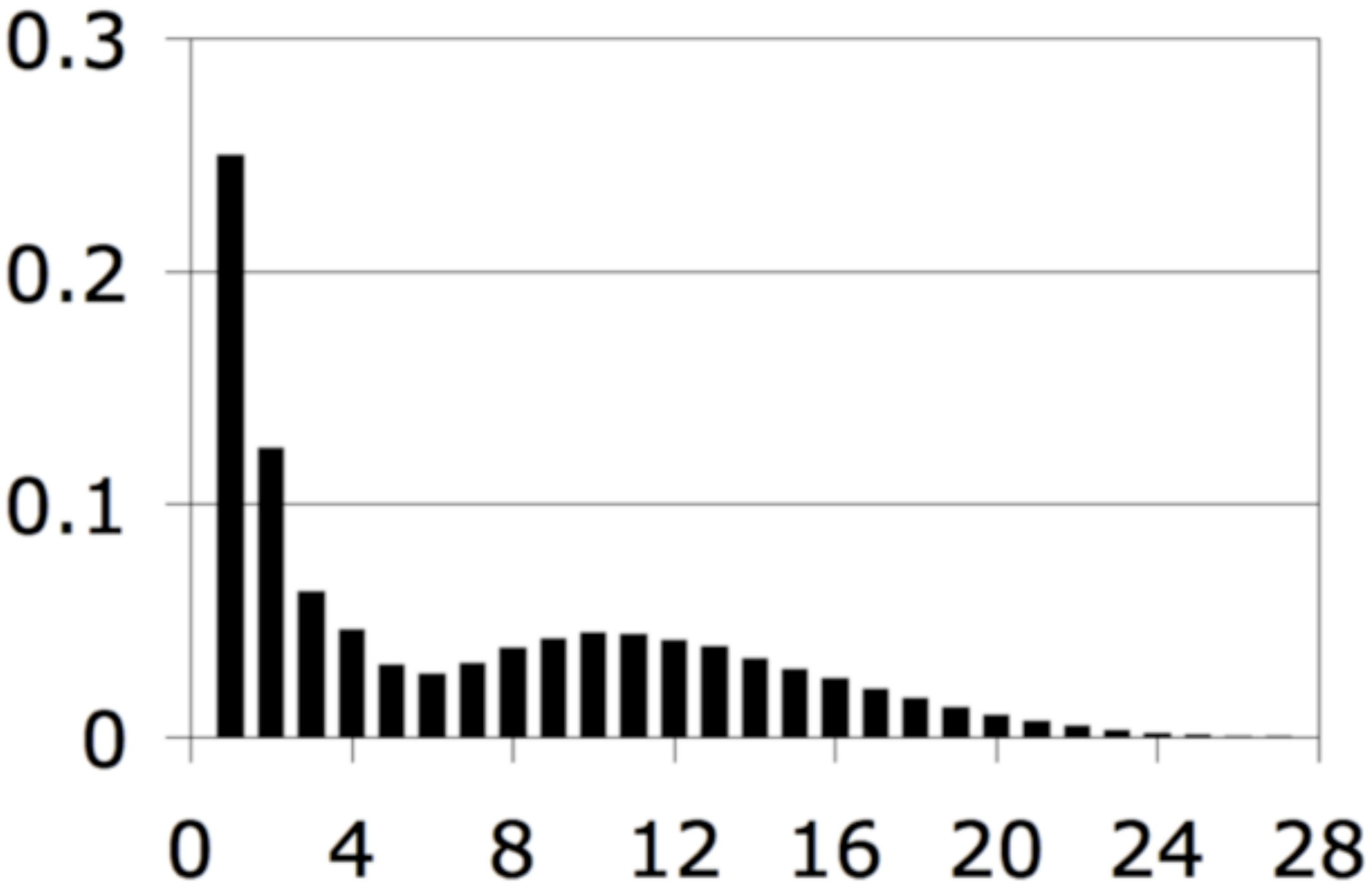}} &
\scalebox{0.17}{\includegraphics*[viewport = 0in 2.7in 8.5in 8.25in]{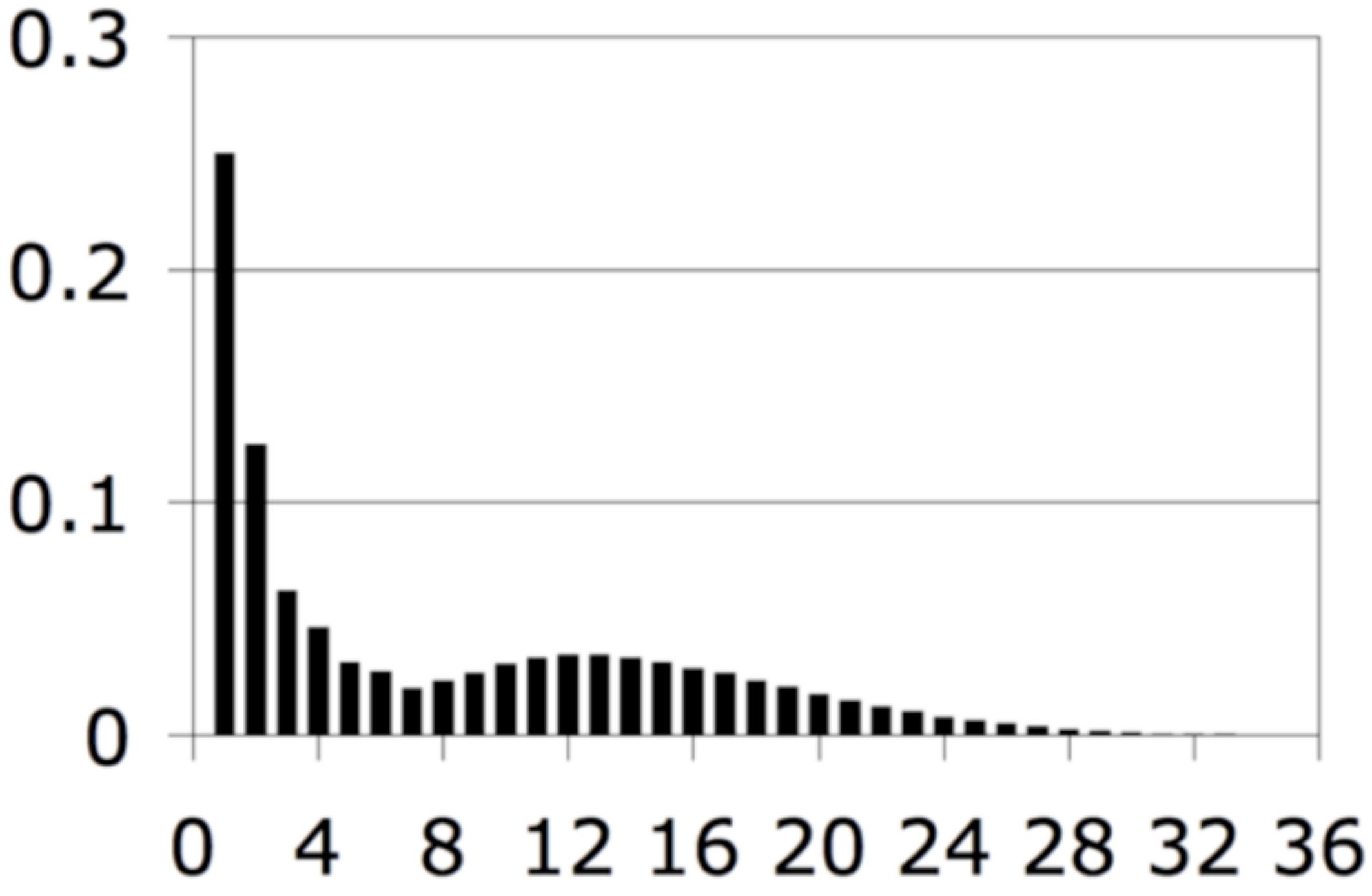}} \\
\end{tabular}
\end{center}
\caption{Load distribution for $k = 5, 7, 8$ respectively.}
\label{loadfig}
\end{figure*}

We present three structural theorems regarding the load distribution.  The first gives basic information needed to make calculations, and the second gives us the value of the first moment of the distribution.  The third discusses a limiting measure for the family of loads that do not originate at the top.  Because of the simplicity of the first theorem, we state it without proof.

\begin{theorem}
\label{loadthm1}
Basic properties of the load distribution.  Let $n_0$ be a fixed time and let $v$ be a node at depth $k$.

\renewcommand{\theenumi}{\alph{enumi}}
\begin{enumerate}

\item All random variables $d_v(n_0)$ are i.i.d. with probability $\frac{1}{2}$ of being L or R.

\item The distribution of $I_v(n_0)$ is laterally translation invariant (i.e. along the x-axis) and is invariant in time for $n_0 \geq k$.

\item The distribution of $I_v(n_0)$ is equal to the distribution of $|C_{w,n_0}^+|$ for any node $w$ with depth equal to min$(n_0, k)$.  Therefore $I_v(n_0)$ takes values in $[1,\frac{n_0(n_0+1)}{2}]$.

\end{enumerate}
\end{theorem}

\begin{theorem}
\label{loadthm2}
Let $v$ be a node at depth $k$.  The mean of the load distribution is

\[ {\mathbb E} (I_v(n)) = 
\begin{cases}
n & n \leq k \\
k & n > k 
\end{cases} \]

\end{theorem}

\begin{proof}
We prove by induction on $k$.  For $k=1$, the statement is trivial, so consider $k>1$.  Since the distribution of $I_v(n)$ is constant for $n \geq k$, we assume $n \leq k$.  Let $N_{v,k-1}$ be the number of nodes at level $k-1$ which send sediment to $v$ at the end of time $n-1$.  This variable takes values in $\{0,1,2\}$ with probabilities $\{1/4, 1/2, 1/4\}$, respectively.  Call $w_1$ $(w_2)$ the left (right) parent of $v$.

\[ {\mathbb E}(I_v(n)) = \sum_{i=1}^2 {\mathbb E}(I_v(n) | N_{v,k-1}=i) {\mathbb P}(N_{v,k-1}=i) \]

\[ = 1/4 + 1/2 \left[ 1+ {\mathbb E}(I_{w_1}(n-1)) \right] + 1/4 \left[ 1+ {\mathbb E}(I_{w_1}(n-1) + I_{w_2}(n-1)) \right] \]

\[ = 1 + {\mathbb E}(I_{w_2}(n-1)) = 1 + (n-1) = n \]
where to go from the second line to the third line, we use the fact that the variables $I_{w_1}(n-1)$ and $I_{w_2}(n-1)$ have the same distribution (see b. under Theorem \ref{loadthm1}).
\end{proof}

Theorem \ref{loadthm1} lets us use geometric properties of clusters of a static network ($\omega_{n_0}$) to study something which is dynamic: the load at time $n$ at node $v$.  That load may have come from a pathway that no longer even exists at time $n$.  We further exploit this relationship, but to do this we must consider the concept of the dual web, defined in, for example, \cite{FNRS}, and of whose definition we remind the reader.  Consider the odd sublattice

\[ {\mathbb Z}_{odd}^2 = \{ (x,y) \in {\mathbb Z}^2 : x+y \textrm{ odd and } y \leq 1 \} \]

For any node $v^* = (x^*,y^*) \in {\mathbb Z}_{odd}^2$ we call the node $(x^*+1,y^*+1)$ the right child of $v^*$ and we call the node $(x^*-1,y^*+1)$ the left child of $v^*$.  Similarly, we call the node $(x^*+1, y^*-1)$ the right parent of $v^*$ and we call the node $(x^*-1,y^*-1)$ the left parent of $v^*$.  We define the set ${\mathbb E}_{odd}^2$ in the obvious way.  The set of configurations $\{ \omega_n : n \geq 0 \}$ induces a set of configurations $\{ \omega_n^* : n \geq 0 \} \subset \{0,1\}^{{\mathbb E}_{odd}^2}$ by the following rule.  If, in the configuration $\omega_n$, a node $v=(x,y)$ is connected to its left child, we form a connection between the node $v^* = (x, y-1)$ and its right child in the configuration $\omega_n^*$ by setting the image under $\omega_n^*$ of the edge in ${\mathbb E}_{odd}^2$ between $v^*$ and its right child to 1, and the image of the edge between $v^*$ and its left child to 0.  If, on the other hand, $v$ is connected to its right child in $\omega_n$ then we set the image of the edge from $v^*$ to its left child under $\omega_n^*$ to 1 and the image of the edge from $v^*$ to its right child to 0.  See Figure \ref{dualwebfig} and notice that we construct clusters in $\omega_n^*$ so that no occupied edges in $\omega_n^*$ cross any occupied edges in $\omega_n$.

\begin{figure*}
\begin{center}
\scalebox{0.35}{\includegraphics*[viewport = 0in 2.2in 8.5in 8.8in]{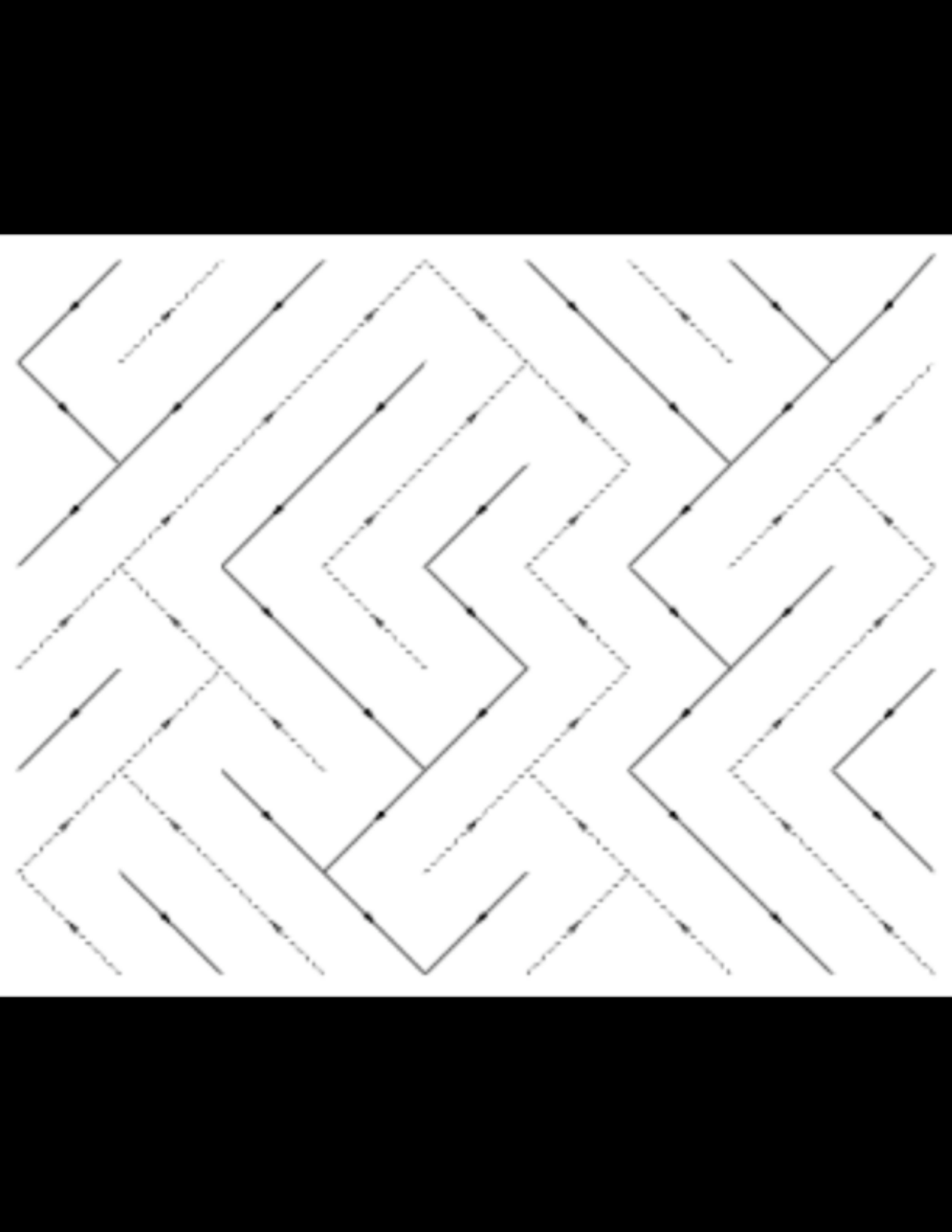}} 
\end{center}
\caption{Portion of an realization of the erosion network, along with its dual web.  The solid lines indicate paths of sediment flow and the dotted lines show paths of the dual (courtesy of \cite{FINR}).}
\label{dualwebfig}
\end{figure*}

The upward paths in $\omega_n^*$ now resemble the downward paths in $\omega_n$.  That is, the upward path starting at the node $v^*$ is a simple symmetric random walk which is killed at depth 1.  Random walks starting at different nodes are independent until they meet, at which point they coalesce into one random walk.  (This is similar to the coalescing random walks picture of the discrete web, described in \cite{arratia}, \cite{howittwarren}, \cite{FNRS}.)

There is an obvious physical interpretation for the paths in the dual web.  For any two adjacent paths in the configuration $\omega_n$, there is a path in $\omega_n^*$ separating them.  If the paths in $\omega_n$ represent rills or drains, the paths in $\omega_n^*$ represent the divides or ridges between them.  Just as divides between rills do not cross rills, paths in $\omega_n^*$ do not cross paths in $\omega_n$.

We now characterize the load distributions for our model.  For any node $v=(x,y)$ (with depth $k$), let $v_L^*=(x-1,y)$ and let $v_R^*=(x+1,y)$.  Consider the set of edges in the dual lattice contained in the paths emanating from the vertices $v_R^*$ and $v_L^*$ in $\omega_n^*$ until either (a) they meet at some vertex $w^*$ or (b) they reach a depth of 1.  The set of nodes in ${\mathbb Z}_{even}^2$ in the interior of this set of edges is exactly the backward cluster of $v$ in the configuration $\omega_n$.  

We now make some definitions so that we can work with this load distribution.  Let $\{ X_i^L : i \geq 2 \}$ and $\{ X_i^R : i \geq 2 \}$ be independent sets of random variables (also independent of each other) which take the values 1 and -1 each with probability $\frac{1}{2}$.  For $i \geq 2$ let $Y_i =\frac{1}{2}( X_i^R - X_i^L)$ and for $i \geq 1$, let $W_i = 1 + Y_2 + ... + Y_i$.  Consider the stopping time 

\[ \tau = \min \{ n: W_n = 0 \} \]
Up until the stopping time $\tau$, the random variable $W_i$ represents the width of the backward cluster of the node $v$ in the real lattice (only valleys and not separating ridges), where we only consider nodes in this cluster whose depths are between $k -i +1$ and $k$.  Therefore the total number of nodes in this partial cluster should be 

\[ L_i := W_1 + ... + W_i \]

Now we can make an equivalent definition of the distribution of the load  $I_v(n)$ by saying that for each fixed $n$, it is the same as the distribution of the random variable 

\begin{equation}
\label{equivdefeq}
L_k(n) := L_{\min(\tau, n, k)}
\end{equation}
This variable is essentially a discrete integral of the symmetric random walk $\{ W_i : i \geq 1 \}$.

\begin{theorem}
\label{loadthm3}
Let $v$ be a node at depth $k_v$ and let $w$ be a node at depth $k_w \geq k_v$.  For any $n \geq k_v$ and for any $l < k_v$ we have 

\[ {\mathbb P}(I_v(n) = l) = {\mathbb P}(I_w(n) = l) \]
Therefore the limit

\begin{equation}
\label{loaddistlimiteq}
\lim_{k_v \to \infty} I_v(k_v)
\end{equation}
exists in distribution.  This limit is a.s. finite but has infinite mean.
\end{theorem}

\begin{proof}
On the event $\tau \geq k_v$, i.e. the load originated from the top,

\[ I_v(n) = L_{\min(\tau,n,k_v)} = L_{k_v} \geq k_v > l \]
and

\[ I_w(n) = L_{\min(\tau,n,k_w)} \geq L_{\min(\tau,n,k_v)} > l \]
Hence, we need only consider $\tau < k_v$.

\[ {\mathbb P}(I_v(n) = l) = {\mathbb P}(I_v(n) = l, \tau < k_v) = {\mathbb P}(L_{\min(\tau,n)} = l, \tau < k_v) \]

\begin{equation}
\label{loadeq1}
= {\mathbb P}(L_{\min(\tau,n,k_w)} = l, \tau < k_v) = {\mathbb P}(I_w(n) = l)
\end{equation}

The random variable $L_{k_v}(n)$ is constant for $n \geq k_v$, so
\[ {\mathbb P}(I_v(k_v)=l) = {\mathbb P}(I_v(k_w) = l) = {\mathbb P}(I_w(k_w)=l) \]
where in the last equality we use (\ref{loadeq1}).  Consequently, for any fixed $l$, the limit

\[ \lim_{k_v \to \infty} {\mathbb P}(I_v(k_v) = l) \]
exists.  By the definition (\ref{equivdefeq}), a random variable with this limiting distribution is

\[ L_{\infty} := \lim_{k \to \infty} L_{\min(\tau,k,k)} = L_{\tau} \]
Since 

\[ \tau \leq L_{\tau} \leq \frac{\tau(\tau+1)}{2} \] 
the third statement of the theorem will follow if we show that $\tau$ is a.s. finite and has infinite mean.  But since the increments $\{Y_i\}$ of the random walk $\{W_i\}$ have mean zero, the walk is recurrent.  In addition, it is a standard result that the entrance time of the set $\{0\}$ has infinite mean.  This completes the proof.
\end{proof}

\subsection{Dynamics}
Now we investigate some aspects of the effect of $\eta$ on the stability of configurations over time.  As noted, the dynamics creates an interpolation between the discrete web and the discrete dynamical web.  The evolution of the system mirrors some aspects of rill erosion, one being that nodes through which a large amount of water passes at time $n_0$ have a non-trivial probability to channel a large amount of water at any time $n_1 > n_0$.  The degree to which this is true depends on the parameter $\eta$, as we will see.  

\subsubsection{Load Correlation}
We start our analysis by inspecting simulation results.  For any two positive integers $M,N$, let $V_{M,N}$ be an enumeration of the $MN$ nodes in the box $[0,M-1] \times [-N,-1]$ and define the load correlation coefficient at time $n$ (for $n \geq N$) by 

\[ K_{M,N}(n) = \frac{\sum_{v \in V_{M,N}} I'_v(N)I'_v(n)}{\sqrt{(\sum_{v \in V_{M,N}} I'_v(N)^2)(\sum_{v \in V_{M,N}} I'_v(n)^2)}} \]
where $I'_v(n) = I_v(n) - {\mathbb E}(I_v(n))$.  This quantity is only defined for $n \geq N$ because two load vectors for a box of depth $N$ are in some sense incomparable if they are taken at times $n_0, n_1$ with $n_0 < N \leq n_1$.  For example, a node at depth $n$ only has a maximum possible load of $\frac{n(n+1)}{2}$ at time $n<N$, whereas its maximum possible load is $\frac{N(N+1)}{2}$ for $n \geq N$.

In Figure \ref{loadcorfig} we have graphed simulation results for a network 49 nodes deep and 49 nodes wide.  The x-axis represents values of the parameter $\eta$, as it varies from 0 to 5.  The y-axis represents values of a time averaged correlation coefficient, namely the quantity

\[ \frac{1}{N'-N} \sum_{n = N}^{N'} K_{M,N}(n) \]
for $M = N = 49$ and $N' = 200$.  Furthermore, we averaged this value over 6 independent trials.  This quantity is meant to approximate values of $K_{M,N}(n)$ for $M,N,n$ large.  The time averaging seemed necessary because of fluctuations, most likely due to finite size conditions, in the quantity $K_{M,N}(n)$.

\begin{figure*}
\begin{center}
\scalebox{0.5}[.4]{\includegraphics*[viewport = 0in 2.7in 8.75in 8.3in]{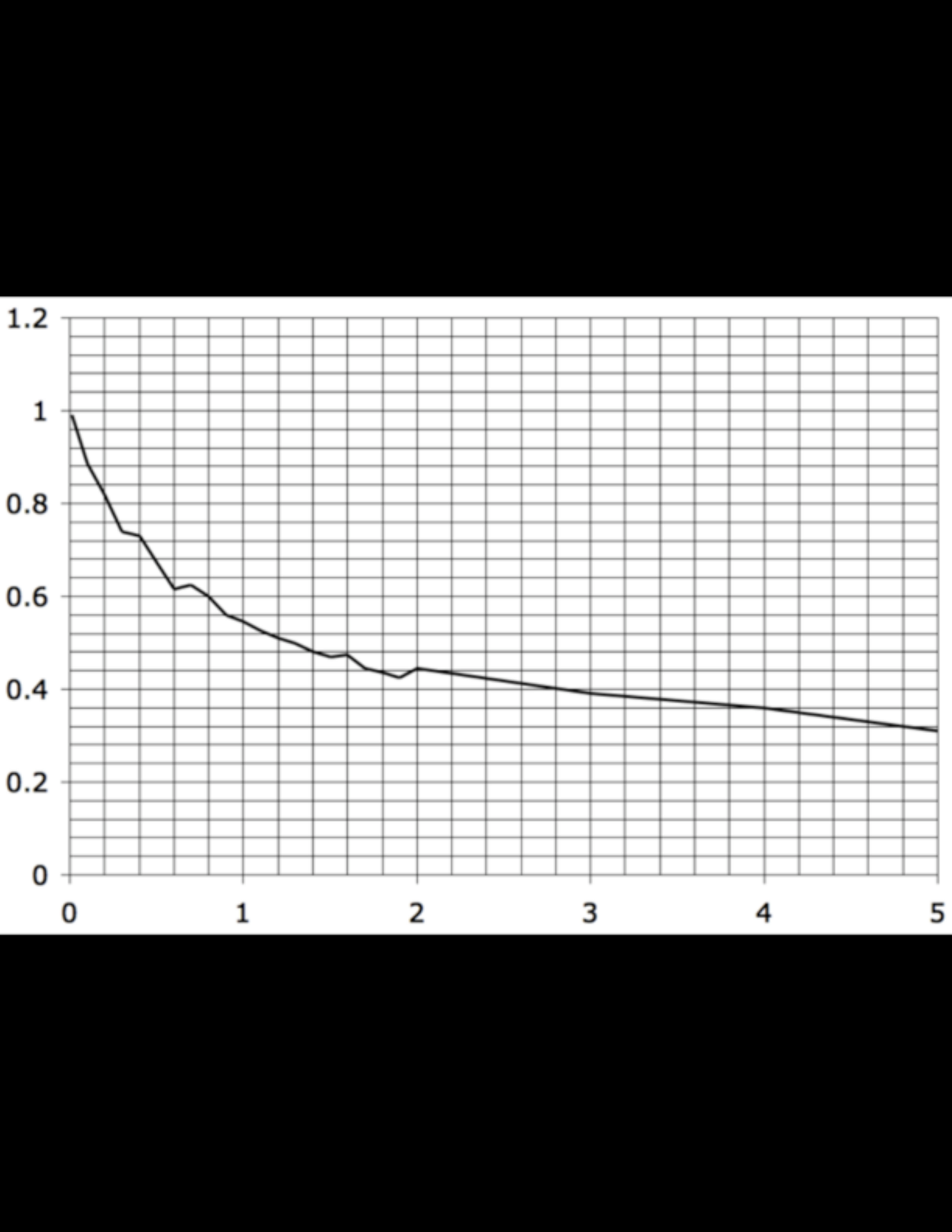}}
\end{center}
\caption{Averaged load correlation (vertical axis) versus $\eta$ (horizontal axis).  Data points are taken at $\eta$-intervals of 0.1 until $\eta=2$, and then at intervals of 1.}
\label{loadcorfig}
\end{figure*}

We can see in Figure \ref{loadcorfig} that the coefficient approaches 1 as $\eta \to 0$.  This makes sense because, as remarked in section 1.3, the $\eta \to 0$ limit of the dynamics (in any fixed box) is the same as the dynamics (or rather non-dynamics) of the discrete web.  Therefore the load vector for this box should be similar (if not the same) at any two times.  As $\eta$ increases, the correlation coefficient decreases and appears to approach 0.  Indeed, additional simulations give the following data: for $\eta = 10, 100, 1000, 10000$, the coefficients were $.2391, .0896, .0189,$ and $.0101,$.  From the discussion of the $\eta \to \infty$ limit given in section 1.3, the correlation coefficient should approach that computed from two load vectors from independent realizations of the discrete web.

\subsubsection{de Finetti Measures}

Whereas we can compare nodes at the top level to standard P\'{o}lya urns, we can compare lower level nodes to \textit{time-dependent input} \cite{pemantle} or \textit{random input} \cite{CAP} P\'{o}lya urns.  We start with an urn with $R_0$ red balls and $B_0$ black balls, as before, but we also have a time-dependent (or random) input sequence $I=(I_0, I_1, ...)$.  At time $t=n$ we draw a ball from the urn and we return it to the urn along with $I_n$ balls of the same color.  Notice that this process with $I=(1,1,...)$ is just the standard P\'{o}lya urn.  

To analyze these lower level nodes, we will also make use of a fundamental result in the theory of exchangeable variables.  

\begin{definition}
$\{0,1\}$-valued variables $X_1, X_2, ...$ are \textbf{exchangeable} if for any $x_1, ..., x_m \in \{0,1\}$ and for any permutation $\sigma$ of $m$ elements we have 

\[ {\mathbb P} (X_1 = x_1, ..., X_m = x_m) = {\mathbb P} (X_{\sigma(1)} = x_1, ..., X_{\sigma(m)} = x_m) \]
\end{definition}

\begin{theorem}[de Finetti]
\label{definettithm}
Let $(\Omega, {\cal F}, {\mathbb P})$ be a probability space and suppose that $\{X_n\}_{n \geq 0}$ are exchangeable $\{0,1\}$-valued random variables defined on $\Omega$.  Then there exists a random variable $F$ on $\Omega$ so that conditioned on $F$, the random variables $X_n$ are independent Bernoulli with parameter $F$.
\end{theorem}

It is easy to verify that if depth$(v) = 1$, then the $\{0,1\}$-valued variables $\{D_v^L(n)\}_{n \geq 0}$ are exchangeable.  In our case, the variable $F$ from Theorem \ref{definettithm} is actually 

\begin{equation}
\label{pvdefeq}
p_v := \lim_{n \to \infty} P_v^L(n).
\end{equation}
Thus if we know the asymptotic fraction of left choices for a node, then our node is just flipping independent coins each second with the same bias.

At lower levels, the variables $\{D_v^L(n)\}$ are not exchangeable.  However, they are asymptotically exchangeable.  We use the definition of Kingman \cite{kingman}.

\begin{definition}
$\{0,1\}$-valued random variables $X_1,X_2,...$ are called \textbf{asymptotically exchangeable} if there exists a sequence $Y_1,Y_2,...$ of exchangeable random variables so that for each $x_1,...,x_m \in \{0,1\}$,

\[ \lim_{N \to \infty} {\mathbb P}(X_{1+N} = x_1, ..., X_{m+N} = x_m) \]

\[ = {\mathbb P}(Y_1 = x_1, ..., Y_m = x_m) \]
In the language of Theorem \ref{definettithm}, let $F$ be the random variable associated with the exchangeable variables $\{X_n\}$.  We call $F$ the \textbf{de Finetti measure} for the sequence $\{Y_n\}$.
\end{definition}

Let $v$ be a node with depth $k \geq 1$.

\begin{theorem}
\label{dynamicsthm1}
The variables $\{P_v^L(n)\}_{n \geq 0}$ form a bounded martingale sequence w.r.t. ${\cal F}_n$.  Therefore they have an almost sure limit $p_v$.
\end{theorem}

\begin{proof}
Similar to the proof of \cite[Theorem 2.1]{CAP}
\end{proof}

\begin{remark}
Using the same equations which produce (\ref{pequalsl}), the limit in Theorem \ref{dynamicsthm1} is the same as the limit of the variables $\{LF_v^L(n)\}_{n \geq 0}$.  
\end{remark}

For any number $0 \leq p \leq 1$, define the measure ${\mathbb Q}_p$ on the set $\{0, 1\}$ by

\[ {\mathbb Q}_p(\{0\}) = 1-p \textrm{ , } {\mathbb Q}_p(\{1\}) = p \]
Let $\{v_1, ..., v_r\}$ be a finite set of vertices.  For any vector of real numbers $(p_1, ..., p_r)$, each between 0 and 1, define the product measure ${\mathbb Q}_{\vec{p}}$ on vectors in $\{0,1\}^r$ to be the product measure $\prod_{i =1}^r {\mathbb Q}_{p_i}$.

\begin{theorem}
\label{dynamicsthm2}
For fixed $v$, the variables $\{D_v^L(n) : n \geq 1\}$ are asymptotically exchangeable with de Finetti measure equal to the distribution of $p_v$.  Furthermore, let $\vec{v} = (v_1, ..., v_r)$ be a vector of vertices and for each $n$, let $D_{\vec{v}}^L(n) = (D_{v_1}^L(n), ..., D_{v_r}(n))$.  If $\vec{d}_1, ..., \vec{d}_s$ are vectors in $\{0,1\}^r$, with probability one,

\[ \lim_{T \to \infty} {\mathbb P}(D_{\vec{v}}^L(1 + T) = \vec{d}_1, ..., D_{\vec{v}}^L(s + T) = \vec{d}_s) = {\mathbb E}(\prod_{i=1}^s {\mathbb Q}_{\vec{p}} (\vec{d}_i)) \]
where $\vec{p} = (p_{v_1}, ..., p_{v_r})$.
\end{theorem}

\begin{proof}
Similar to the proof of \cite[Theorem 2.2]{CAP}.  
\end{proof}

Because of lateral translation invariance, the de Finetti measure for $v$ depends only on the depth $k$.  In light of this, we define

\begin{equation}
\label{definettidef}
\theta_k = \textrm{ de Finetti measure for row } k
\end{equation}
With this framework we will be able to study the switching rate of each node $v$ once we have the following lemma.

\begin{lemma}
\label{freqlemma}
For any node $v$, almost surely,

\begin{equation}
\lim_{n \to \infty} \frac{1}{n} \sum_{i=1}^n D_v^L(i) = p_v
\end{equation}
\end{lemma}

\begin{proof}
Similar to the proof of \cite[Theorem 2.3]{CAP}.
\end{proof}

We now define the switching function $s_v$ for $n \geq 2$ by $s_v(n) = D_v^L(n)(1-D_v^L(n-1)) + D_v^L(n-1)(1-D_v^L(n))$.  Define the switching rate $S_v(n)$ to be the time average of $s_v$, that is

\[ S_v(n) = \frac{1}{n-1} \sum_{i=2}^n s_v(i) \]

\begin{theorem}
The $n \to \infty$ limit of $S_v(n)$ exists a.s. 

\begin{equation}
\lim_{n \to \infty} S_v(n) = 2p_v(1-p_v)
\end{equation}
\end{theorem}

\begin{proof}
The proof is similar to the proof of \cite[Theorem 2.3]{CAP}.  Let $d_n = D_v^L(n)$ and $p_n = P_v^L(n)$.  A straightforward calculation gives

\begin{equation}
\label{condexeq}
{\mathbb E}(d_{n+1}|{\cal F}_n) = p_n \textrm{ , } {\mathbb E}(p_{n+1}|{\cal F}_n) = p_n
\end{equation}
Now,

\[ \lim_{n \to \infty} \frac{1}{n} \sum_{i=2}^n s_v(i) = \lim_{n \to \infty} \frac{1}{n} \sum_{i=2}^n (d_i(1-d_{i-1}) + d_{i-1}(1-d_i)) \]

\[ = \lim_{n \to \infty} \frac{1}{n} \sum_{i=2}^n (d_i + d_{i-1}) -2 \lim_{n \to \infty} \frac{1}{n} \sum_{i=2}^n d_i d_{i-1} = 2p_v - 2 \lim_{n \to \infty} \frac{1}{n} \sum_{i=2}^n d_i d_{i-1} \]
by Lemma \ref{freqlemma}.  We must show that the last limit above is a.s. equal to $p_v^2$.

Define $M_n = \sum_{i=2}^n (d_i-p_n)d_{i-1}$.  $M_n$ is a martingale with respect to ${\cal F}_n$: 

\[ {\mathbb E}(M_{n+1}|{\cal F}_n) = {\mathbb E}(\sum_{i=2}^n (d_i d_{i-1}) - \sum_{i=2}^n (p_{n+1} d_{i-1}) + d_{n+1} d_n - p_{n+1} d_n | {\cal F}_n) \]

\[ = \sum_{i=2}^n(d_i d_{i-1}) - \sum_{i=2}^n (p_n d_{i-1}) + p_n d_n - p_n d_n = M_n \]
where we use both equations in (\ref{condexeq}).  Note also

\[ \lim_{n \to \infty} \frac{1}{n} \sum_{i=2}^n p_n d_{i-1} = p_v \lim_{n \to \infty} \frac{1}{n} \sum_{i=2}^n d_{i-1} \]
which is $p_v^2$, by Lemma \ref{freqlemma}.  Therefore it suffices to show that with probability one,

\begin{equation}
\label{avgtozeroeq}
\frac{M_n}{n} \to 0.
\end{equation}
By summing the series

\[ \frac{1}{n}(M_1 + ... + M_n) = \sum_{i=2}^n \frac{M_i}{i} \left( \frac{i}{n} \right) \]
by parts, it can be shown that (\ref{avgtozeroeq}) will follow once we show that 

\[ \sum_{i=2}^{\infty} \frac{M_{i+1}-M_i}{i} \]
converges.  To this end, define $M'_n = \sum_{i=2}^{n-1} \frac{M_{i+1}-M_i}{i}$.  We leave the reader to verify that $M'_n$ is a martingale.  Using $L^2$-orthogonality of martingale differences,

\[ {\mathbb E}(M'_n)^2 = {\mathbb E}(\sum_{i=2}^{n-1} \left( \frac{M_{i+1}-M_i}{i} \right) ^2)  = {\mathbb E}(\sum_{i=2}^{n-1} \left( \frac{(d_i-p_n)(d_{i-1})}{i} \right) ^2) \]

\[ \leq \sum_{i=2}^n \frac{1}{i^2} < \infty \]
Therefore $M_n'$ is an $L^2$ bounded martingale and converges a.s.  This completes the proof.

\end{proof}

If $p_v \in (0,1)$ then neither of the choices L or R are transient for $v$.  This prompts the question of whether or not the de Finetti measures $\theta_k$ have atoms at 0 or 1.  For any fixed $k$, the answer is no.

\begin{theorem}
\label{noatomsthm}
For each $k \geq 1$, the measure $\theta_k$ has no atoms.
\end{theorem}

\begin{proof}
In \cite[Theorem 4]{pemantle} it is shown that a time-dependent input P\'{o}lya urn's de Finetti measure cannot have atoms if there is a $C$ so that $I_v(n) \leq C$ for all $n$.  For each realization of the dynamics and for each $v$, we have $I_v(n) \leq \frac{k_v(k_v+1)}{2}$ for all $n$.  The result follows.
\end{proof}

\begin{corollary}
\label{transiencecor}
Each node $v$ has a nonzero asymptotic switching rate.  Therefore, for each $v$, the states L and R are recurrent.
\end{corollary}

\begin{proof}
This is a direct consequence of Lemma \ref{freqlemma} and Theorem \ref{noatomsthm}.
\end{proof}

\begin{corollary}
\label{loaddivcor}
With probability one, for each node $v$, the variable $I_v(n)$ takes each value in $[1,\frac{k_v(k_v+1)}2]$ for infinitely many values of $n$. 
\end{corollary}

\begin{proof}
Let $v_1, ..., v_m$ be the $m=\frac{k_v(k_v+1)}{2}-1$ nodes above $v$ which can send sediment to $v$ and let $\vec{d}_1, ..., \vec{d}_{v_k} \in \{0,1\}^m$.  Let $N \geq 1$ and write $D_{\vec{v}}^L(n)$ for the vector $(D_{v_1}^L(n), ..., D_{v_m}^L(n))$.

\[ \lim_{N \to \infty} \lim_{T \to \infty} {\mathbb P} (D_{\vec{v}}^L(T+jk_v + 1) = \vec{d_1}, ..., D_{\vec{v}}^L(T+(j+1)k_v) = \vec{d}_{k_v} \textrm{ for some } 1 \leq j \leq N) \]
equals zero almost surely, by Theorem \ref{dynamicsthm2} and Theorem \ref{noatomsthm}.  But this probability dominates the probability of the event $\{D_{\vec{v}}^L(n+1) = \vec{d}_1, ..., D_{\vec{v}}^L(n+v_k+1) = \vec{d}_{v_k}$ for infinitely many $n \}^c$.  The result follows.
\end{proof}

Here an interesting picture of our network emerges.  On the one hand we may view the system as an infinite lattice (the lower half plane), where each node is a random input P\'{o}lya urn.  The output of the urns at depth $k$  at time $n$ becomes the input of the urns at depth $k+1$ at time $n+1$.  On the other hand, as remarked in section 1.3, we may first sample (non-independently) values $\{p_v : v \in {\mathbb Z}_{even}^2 \}$ from the de Finetti measures $\{ \theta_k : k \geq 1\}$ to create an infinite array.  As time $n$ approaches infinity, the behavior of the system approaches the behavior of the same network in which each node $v$ chooses to send its current load left with probability $p_v$ and right with probability $1-p_v$, independently at each second.  Therefore this picture is of a network of two variables, a realization of values $p_v$ from the de Finetti measures, and realization of dynamics which coincides with the dynamics of a much simpler network.  This second network is an obvious generalization of the Dynamical Discrete Web.

Figure \ref{definettifig} shows histograms for the de Finetti measures $\theta_k$ for $k = 2, 5, 9$ and for values of $\eta = .5, 1, 2$.  One sees that the measures become more biased as $k$ increases (for fixed $\eta$).  In other words, the mass of $\theta_k$ is concentrated on domains closer to 0 and 1 than is the mass of $\theta_{k-1}$.   This would seem to imply that the expected asymptotic switching rate of a node at level $k$ (which is $2p_v(1-p_v)$) must decrease with $k$.  Similarly, if $k$ is fixed and $\eta$ decreases to 0, it seems that the expected switching rate should decrease.   

\begin{figure*}
\begin{center}
\begin{tabular}{c c c}
$\eta = \frac{1}{2}$, row 2 & $\eta = \frac{1}{2}$, row 5 & $\eta = \frac{1}{2}$, row 9 \\
\scalebox{0.17}{\includegraphics*[viewport = 0in 2.65in 8.5in 8.3in]{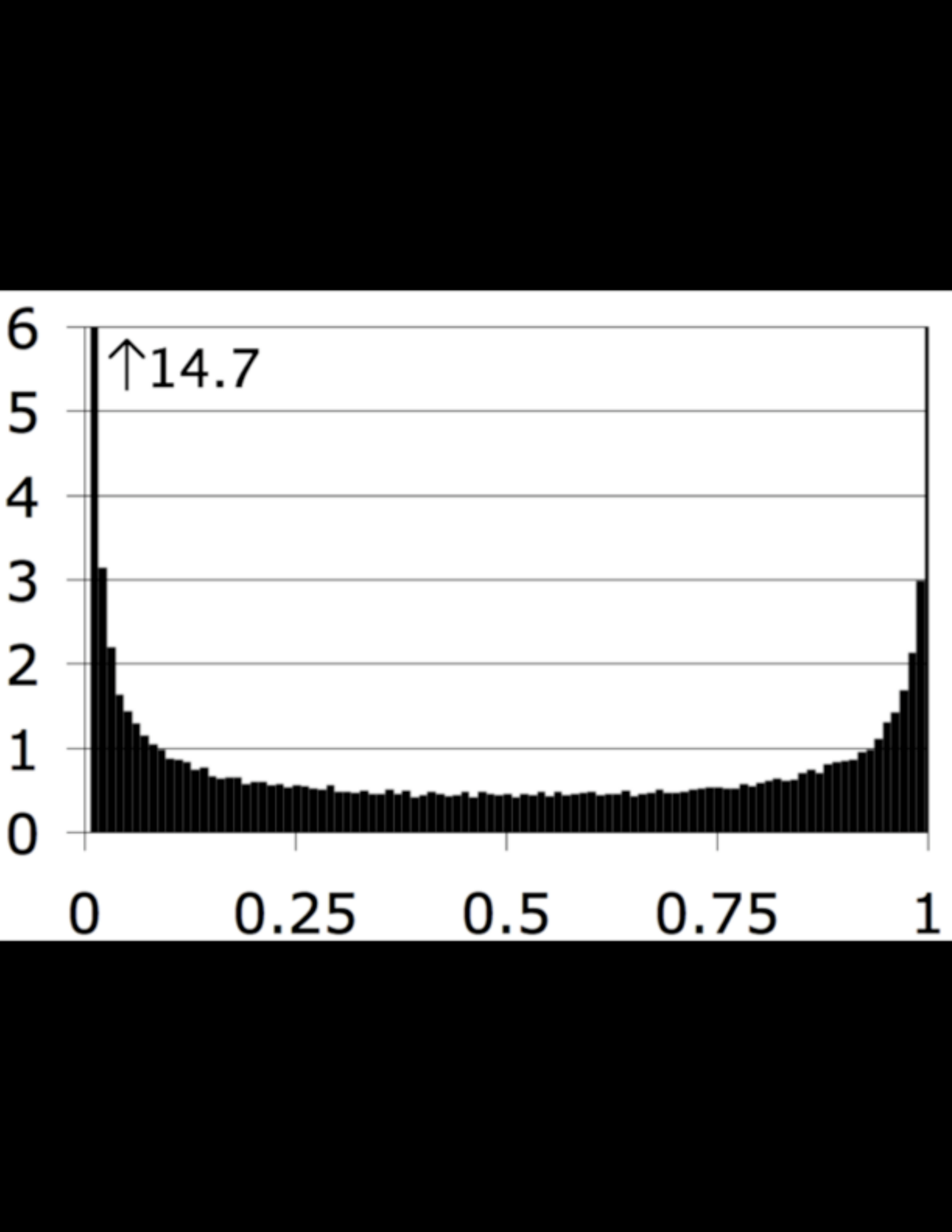}} &
\scalebox{0.17}{\includegraphics*[viewport = 0in 2.65in 8.5in 8.3in]{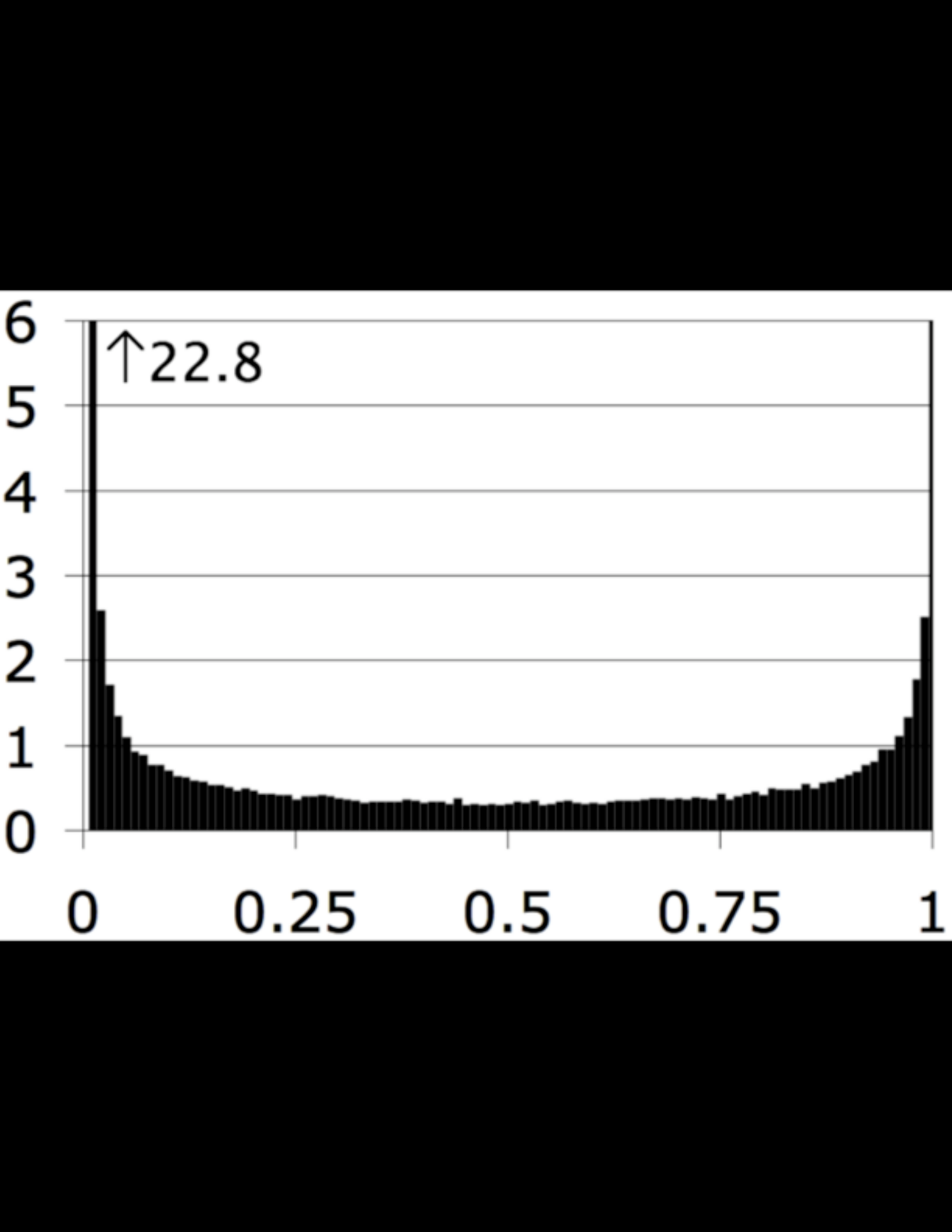}} &
\scalebox{0.17}{\includegraphics*[viewport = 0in 2.65in 8.5in 8.3in]{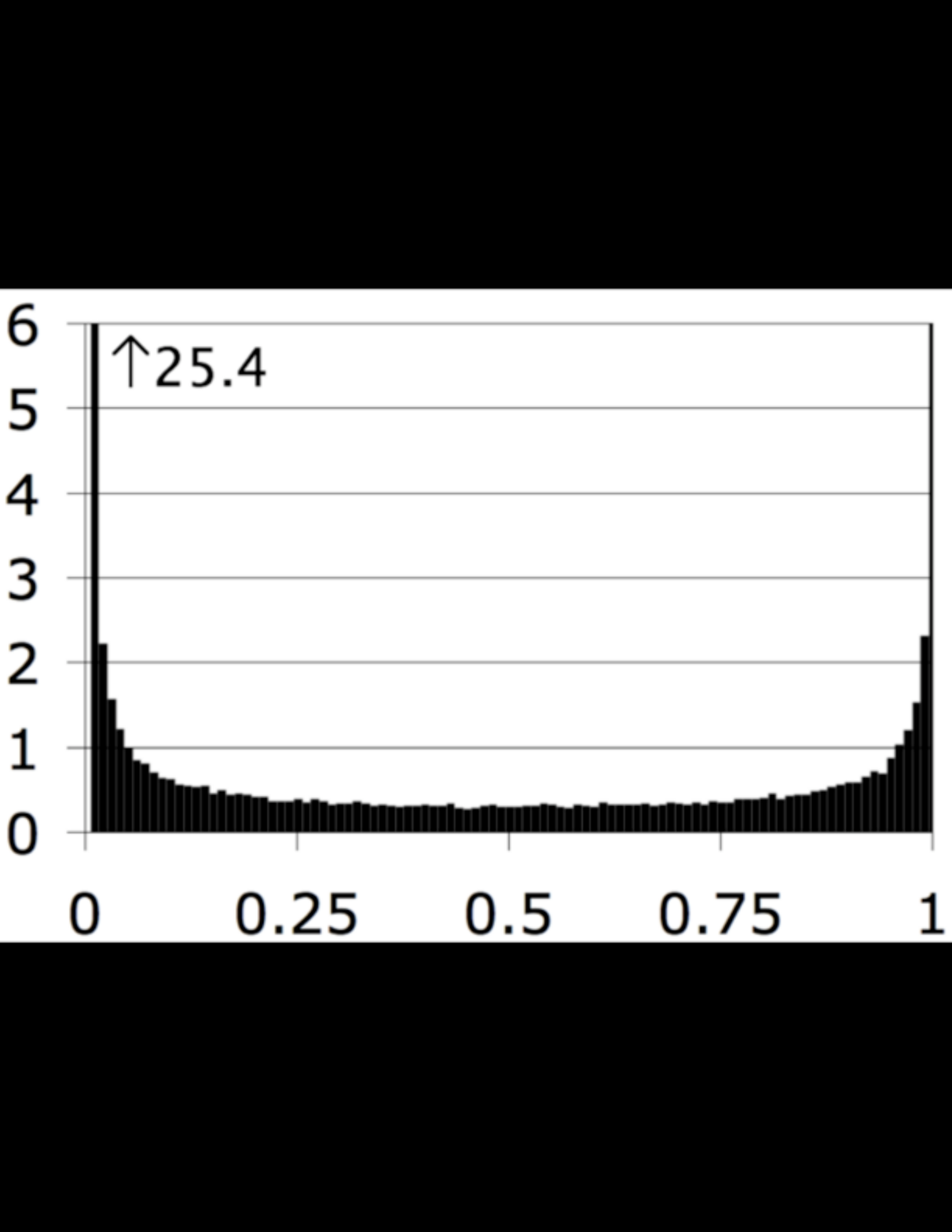}} \\
\\
$\eta = 1$, row 2 & $\eta = 1$, row 5 & $\eta = 1$, row 9 \\
\scalebox{0.18}{\includegraphics*[viewport = 0in 2.87in 8.5in 8.15in]{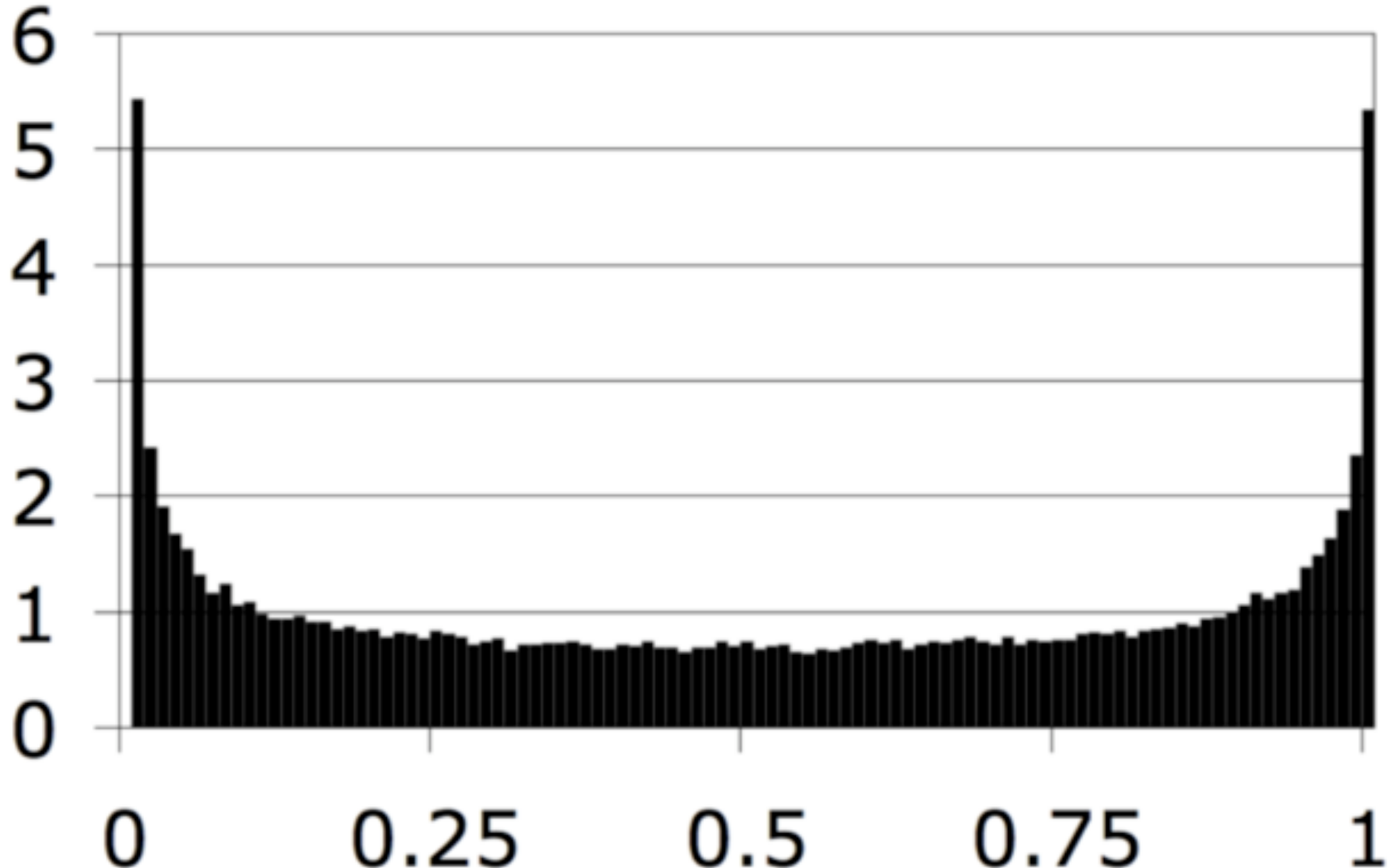}} &
\scalebox{0.17}{\includegraphics*[viewport = 0in 2.65in 8.5in 8.3in]{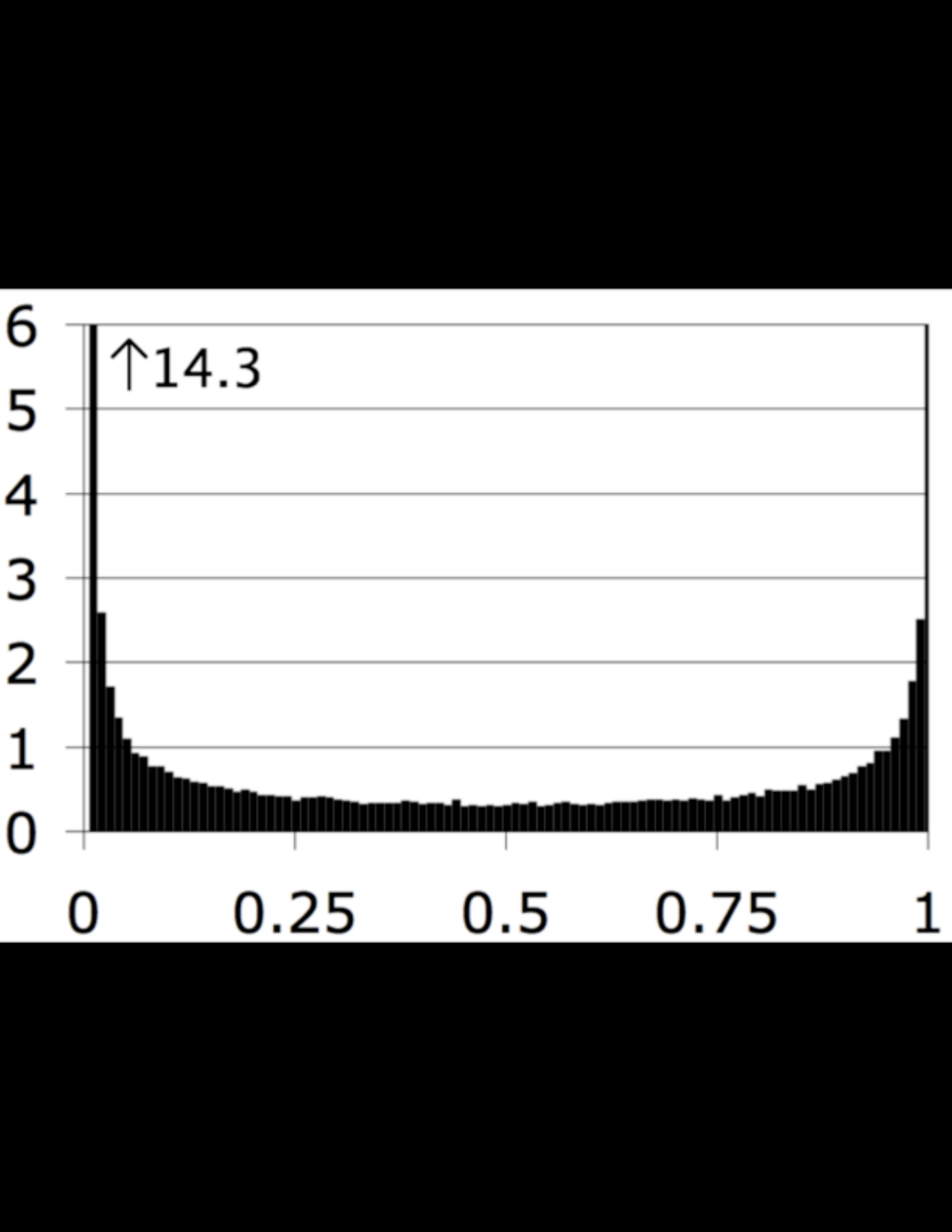}} &
\scalebox{0.17}{\includegraphics*[viewport = 0in 2.65in 8.5in 8.3in]{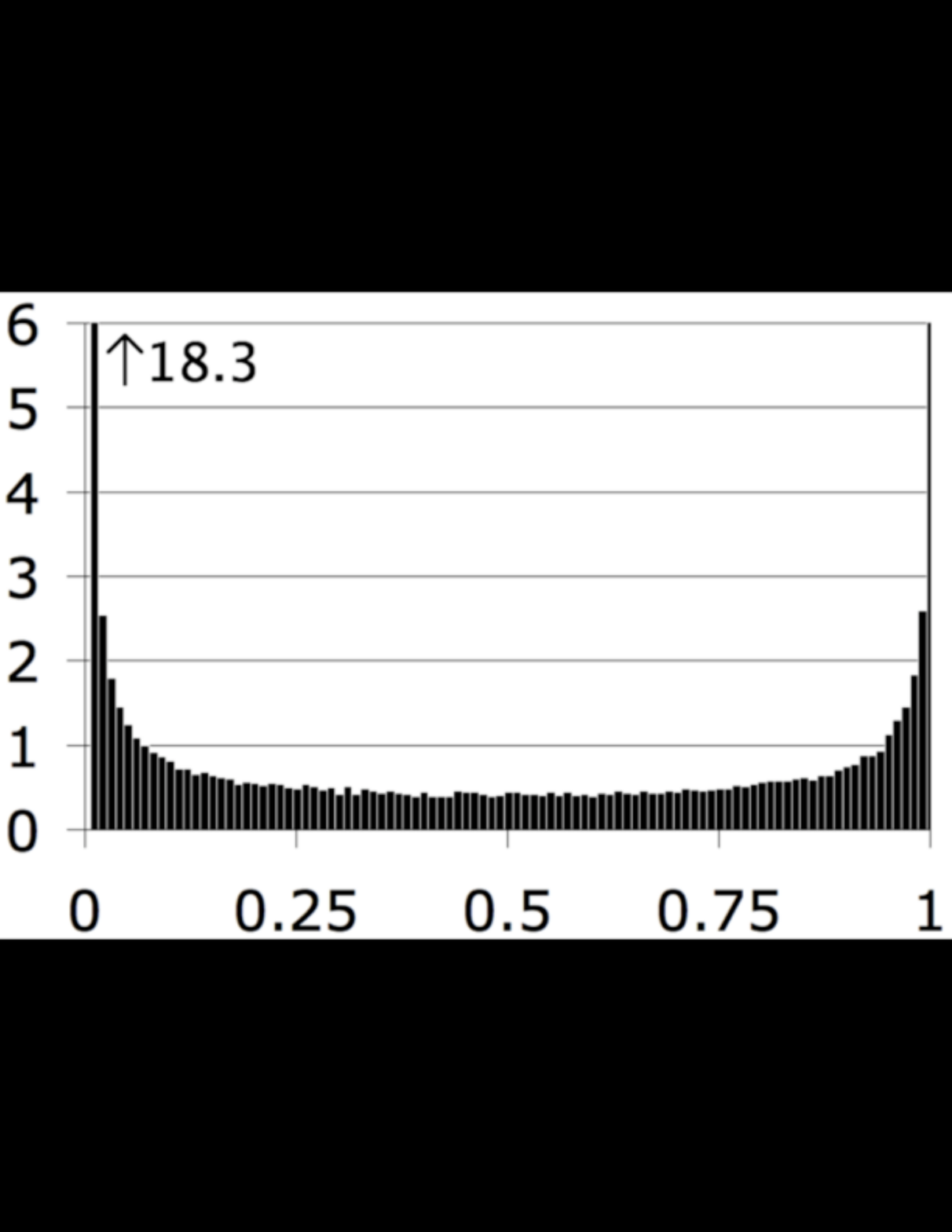}} \\
\\
$\eta = 2$, row 2 & $\eta = 2$, row 5 & $\eta = 2$, row 9 \\
\scalebox{0.17}{\includegraphics*[viewport = 0in 2.65in 8.5in 8.3in]{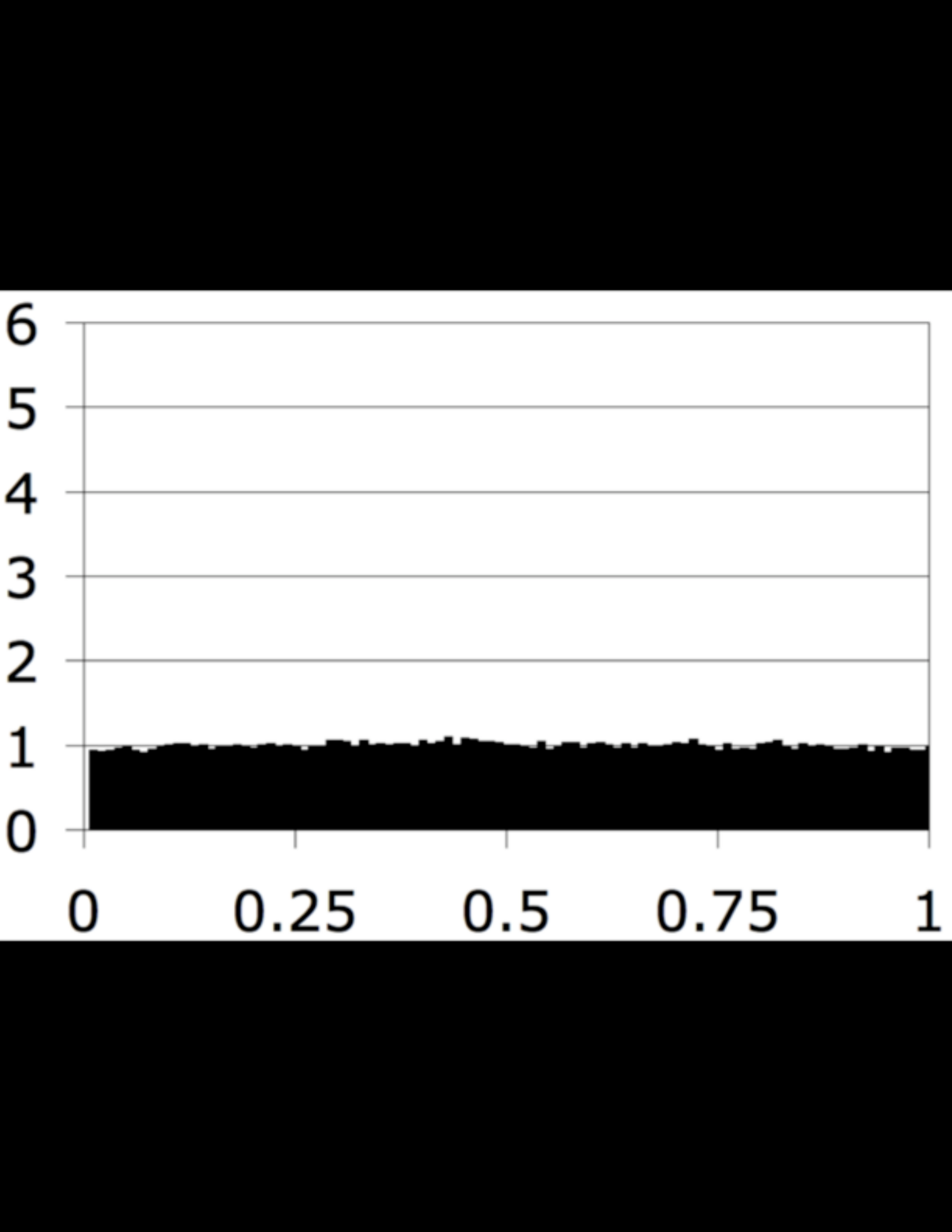}} &
\scalebox{0.17}{\includegraphics*[viewport = 0in 2.65in 8.5in 8.3in]{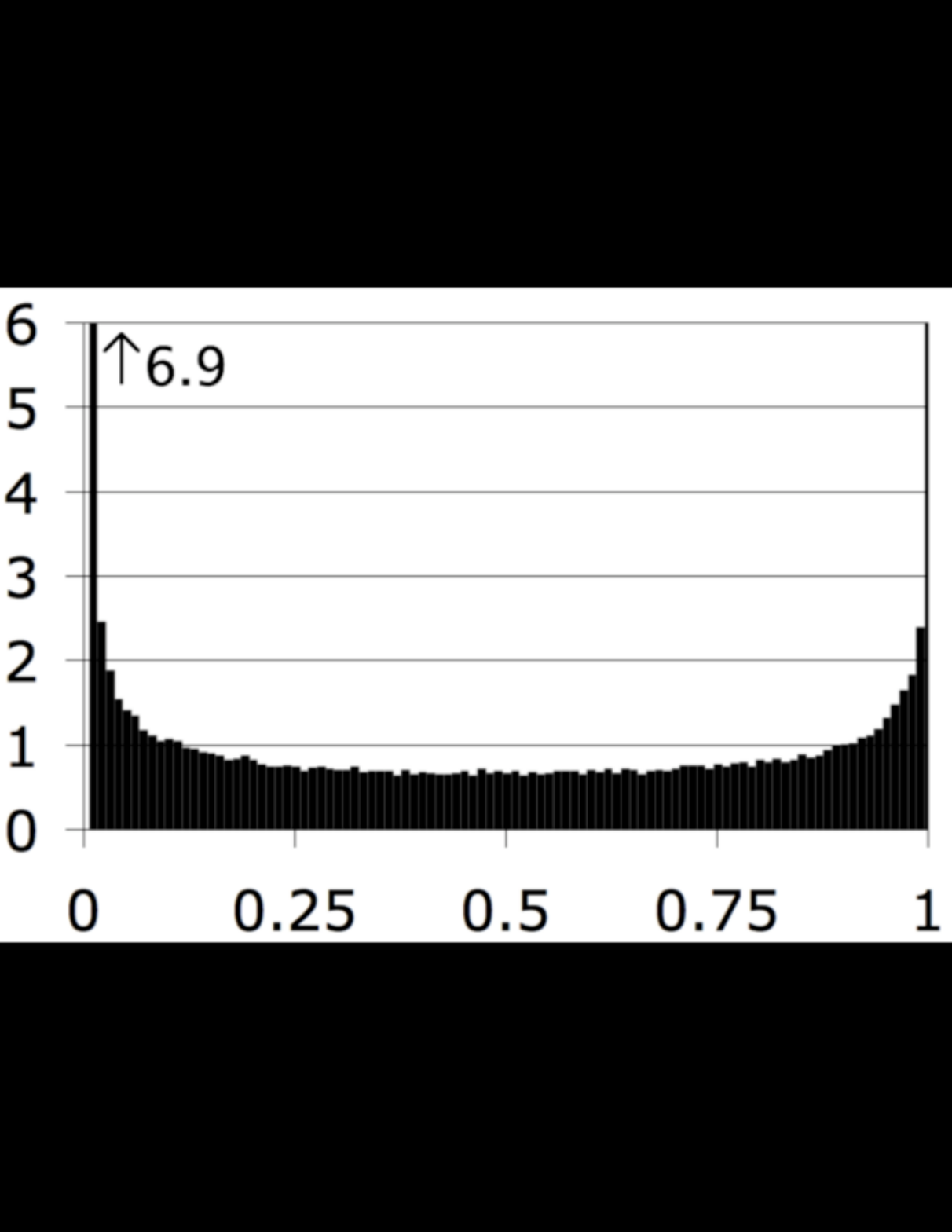}} &
\scalebox{0.17}{\includegraphics*[viewport = 0in 2.65in 8.5in 8.3in]{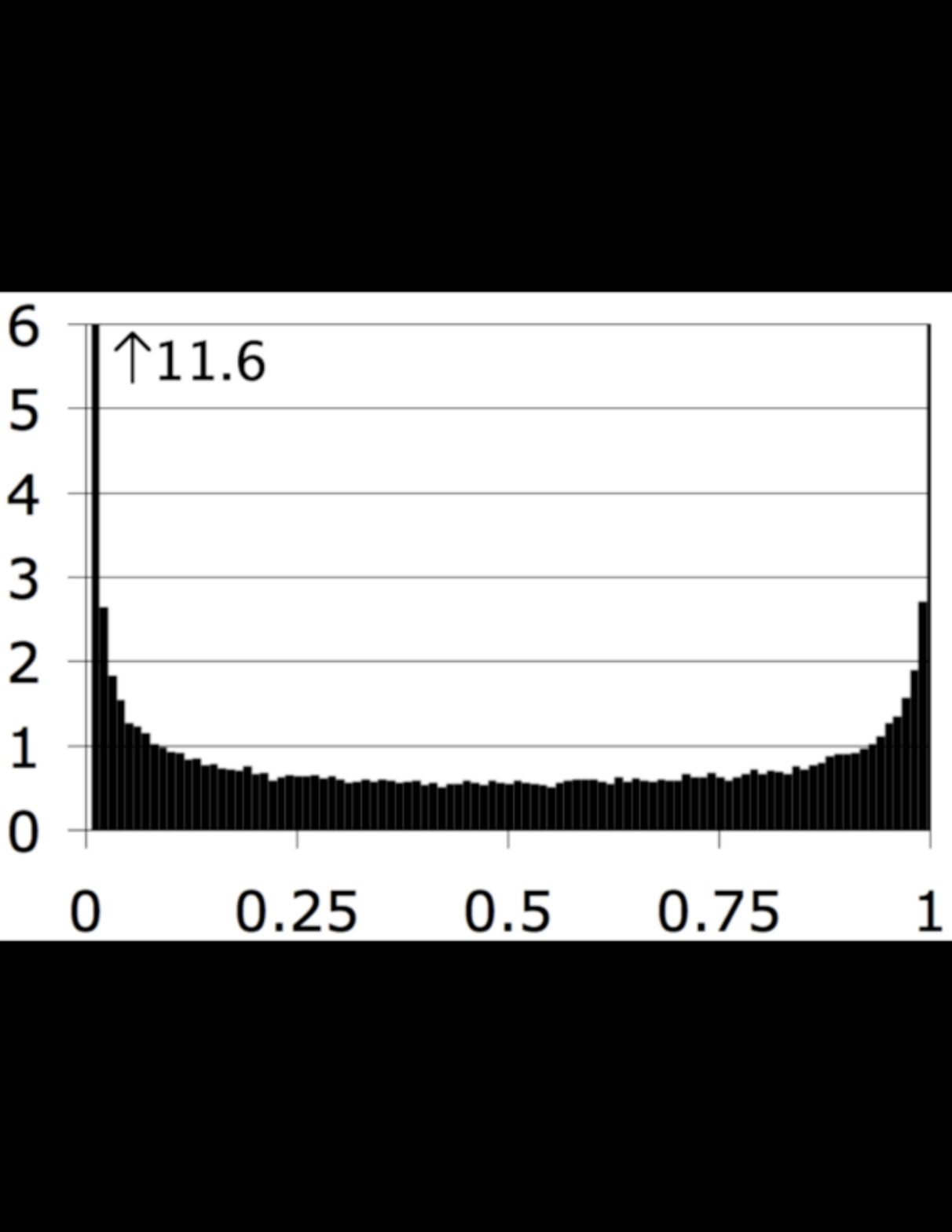}} \\
\end{tabular}
\end{center}
\caption{Definetti measures for $k = 2, 5, 9$ (from left to right) and $\eta = .5, 1, 2$ (from top to bottom).}
\label{definettifig}
\end{figure*}

\begin{figure*}
\begin{center}
\scalebox{0.55}[0.45]{\includegraphics*[viewport = 0in 2.75in 8.5in 8.3in]{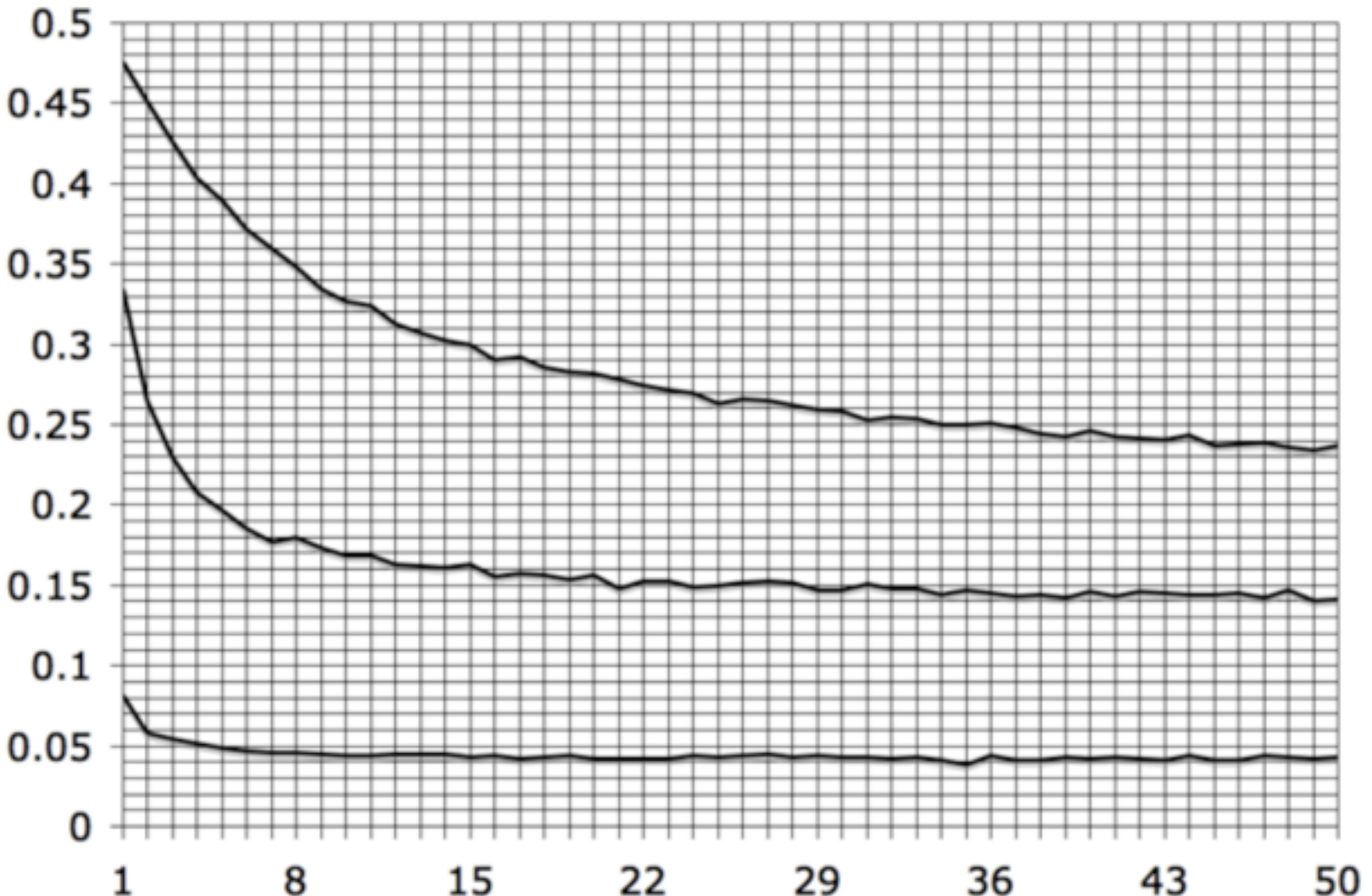}}
\end{center}
\caption{Average switch rate (vertical axis) versus row (horizontal axis).  The values of $\eta$ are 10 (top curve), 1 (middle curve), .1 (bottom curve).}
\label{switchratefig}
\end{figure*}

Figure \ref{switchratefig} represents data given by simulations conducted with an erosion network with width $10^5$, depth 50, and $\eta = .1, 1,$ or 10.  The simulation ran for $n = 1000$ steps and at the end, switch rates for each node in the network were computed.  In each row, each node's rate was averaged.  Since two nodes $v_1$ and $v_2$ with the same depth $k$ have independent behavior as long as they are at least a distance of $2k$ apart, the ergodic theorem gives that, as the network size approaches infinity, the resulting average should resemble the expected switch rate for a row.  The above results results were averaged by row over 3 independent trials.  Finally, the data were plotted by row.  Not only do the average switch rates appear to decrease as $k$ increases, there appears to be a non-trivial (i.e. non-zero and $\eta$ dependent) limit for the switch rate.  This indicates that for at least some values of $\eta$, the limit of the measures $\theta_k$ (if it exists) is most likely not equal to $\frac{1}{2}(\delta_0 + \delta_1)$.

\subsubsection{Catastrophes}

Next, we study the following situation.  Suppose a node $v$ at a large depth $k$ (for this section we assume the depth is at least 2) starts with a small input load and keeps a relatively small input load until a much later time.  Then $v$'s load changes dramatically.  If this new load is sufficiently large, it could bring $v$'s de Finetti measure much closer to $\frac{1}{2}(\delta_0 + \delta_1)$.  This analysis is from the point of view of the node $v$, whereas the analysis of the last half of the section will be from the point of view of the parent.

Let 
\[ A_v(n) = \frac{T_v(n-1)}{n} \textrm{ , } n \geq 1 \]

\begin{definition}
For any $n \geq 1$, define the \textbf{flood ratio} $F_v(n) = \frac{I_v(n)}{A_v(n)}$.  For $c \geq 1$, we say that a \textbf{flood} of order $c$ occurs at time $n$ if $F_v(n) \geq c$.
\end{definition}

\begin{remark} 
Since $I_v(n), A_v(n) \in [1,\frac{1}{2}(k(k+1))]$, we have
\begin{equation}
\label{floodbounds}
\frac{2}{k(k+1)} \leq F_v(n) \leq \frac{k(k+1)}{2}
\end{equation}
\end{remark}

\begin{proposition}
\label{floodprop}
For any $v$, $A_v(n)$ has a limit a.s.  Therefore,

\[ \liminf_{n \to \infty} F_v(n) < \limsup_{n \to \infty} F_v(n) \]
\end{proposition}
\begin{proof}
We show the first statement by induction on depth($k$).  Clearly this is true if depth($v) = 1$.  Otherwise, let $w_1$ be the left parent of $v$ and assume that for all nodes $w'$ with depth equal to that of $w_1$, $A_{w'}(n)$ has a limit.  Let $N_{w_1}^R(n)$ be the number of $i \leq n$ such that $D_{w_1}^R(i) = 0$.  By Lemmas \ref{dynamicsthm1} and \ref{freqlemma}, 

\[ \lim_{n \to \infty} \frac{T_{w_1}^R(n)}{n+1} = \lim_{n \to \infty} \frac{T_{w_1}^R(n) T_{w_1}(n)}{(n+1)T_{w_1}(n)} = (1-p_{w_1}) \lim_{n \to \infty} A_{w_1}(n+1) \]
exists.  The same argument shows that, if $w_2$ is the right parent of $v$, $\lim_{n \to \infty} \frac{T_{w_2}^L(n)}{n+1}$ exists.  Therefore,

\[ \lim_{n \to \infty} \frac{T_v(n-1)}{n} = \lim_{n \to \infty} \frac{n-1 + T_{w_1}^R(n-2) + T_{w_2}^L(n-2)}{n} \]
exists.  

For the second statement of the proposition, we use Lemma \ref{loaddivcor} to see that

\[ \liminf_{n \to \infty} F_v(n) = \frac{1}{\lim_{n \to \infty} A_v(n)} < \frac{k(k+1)}{2 \lim_{n \to \infty} A_v(n)} = \limsup_{n \to \infty} F_v(n) \]

\end{proof}

\begin{remark}
\label{rightavgrmk}
The above proof shows also that $\lim_{n \to \infty} \frac{T_v^R(n)}{n+1}$ exists and equals $(1-p_v)\lim_{n \to \infty}A_v(n)$.
\end{remark}

\begin{remark}
\label{noatomsavgloadrmk}
A simple extension of the above proof using Theorem \ref{noatomsthm} shows that for any level $k$, the distribution of the time-average of the load for level $k$ has no atoms.
\end{remark}

From the previous proposition we know that if $v$ has an asymptotic average input load $A_v$, then $v$ will have infinitely many floods of order $\frac{k(k+1)}{2A_v + \epsilon}$ for any $\epsilon > 0$.  This number can be quite large if $A_v$ is small.  We study numerically the rates of these large floods as they relate to both depth($v$) and $\eta$.

For a fixed node $v$, define the random measure 

\[ \mu_{v,n} = \sum_{i = 2}^n \delta_{F_v(i)} \]
Let $\{v_i\}_{i \geq 1}$ be an enumeration of the nodes with the same depth as that of $v$.  By the ergodic theorem, the average

\[ \frac{1}{M} \sum_{i=1}^M\mu_{v_i,n} \to {\mathbb E}(\mu_{v,n}) \]
as $M \to \infty$.  See Figure \ref{floodfig}.  The graphs come from a simulation run for $10^5$ seconds on a network with a width of $10^4$ nodes and a depth of 50 nodes.  We graphed the density function for the measure $\frac{1}{10^8}\sum_{i=1}^{10000} \left[ \mu_{v_i,10^5} - \mu_{v_i,9000} \right]$ in an attempt to approximate $\frac{1}{n-9000}{\mathbb E}(\mu_{v,n})$ for $n$ large and for $v$ with depth 5, 20, and 50.  The reason we subtracted $\mu_{v_i,9000}$ is to decrease the effect of small times, during which the ratio $F_v(n)$ is likely to be an integer.

As $\eta \to 0$ with a fixed row or as the row increases with fixed $\eta$, each measure seems to concentrate its mass at 0 and 1.  In other words, the measure of any interval which does not include either of these two points appears to approach 0.  In addition, Table \ref{largefloodtable} shows that as $\eta \to 0$ with a fixed row or as the row increases with fixed $\eta$, the expected fraction of time during which a large flood occurs (ratio above 5) increases.  Since the time average of flood ratios approaches 1 as $n \to \infty$, the above facts indicate a trend that as $\eta \to 0$ or as the row increases, the time variance of measures increases, giving more possible variability of the flood ratios.  Although this variance seems to increase, values near 1 (on the x-axis) show that the fraction of time flood ratios spend near 1 increases.  In spirit, this is in accordance with previous results, as we explain.  Figure \ref{switchratefig} shows that as the row increases, the expected switching rate of a node decreases.  It is reasonable to believe that the same conclusion holds if the depth is fixed but $\eta$ decreases.  Therefore the network prefers to be more static in these circumstances and we would expect a node to receive a relatively constant load, forcing flood ratios to be near 1.

\begin{figure*}
\begin{center}
\begin{tabular}{c c c}
$\eta = \frac{1}{10}$, row 5 & $\eta = \frac{1}{10}$, row 20 & $\eta = \frac{1}{10}$, row 50 \\ 
\vspace*{.1in}
\scalebox{0.17}{\includegraphics*[viewport = 0in 2.8in 8.5in 8.2in]{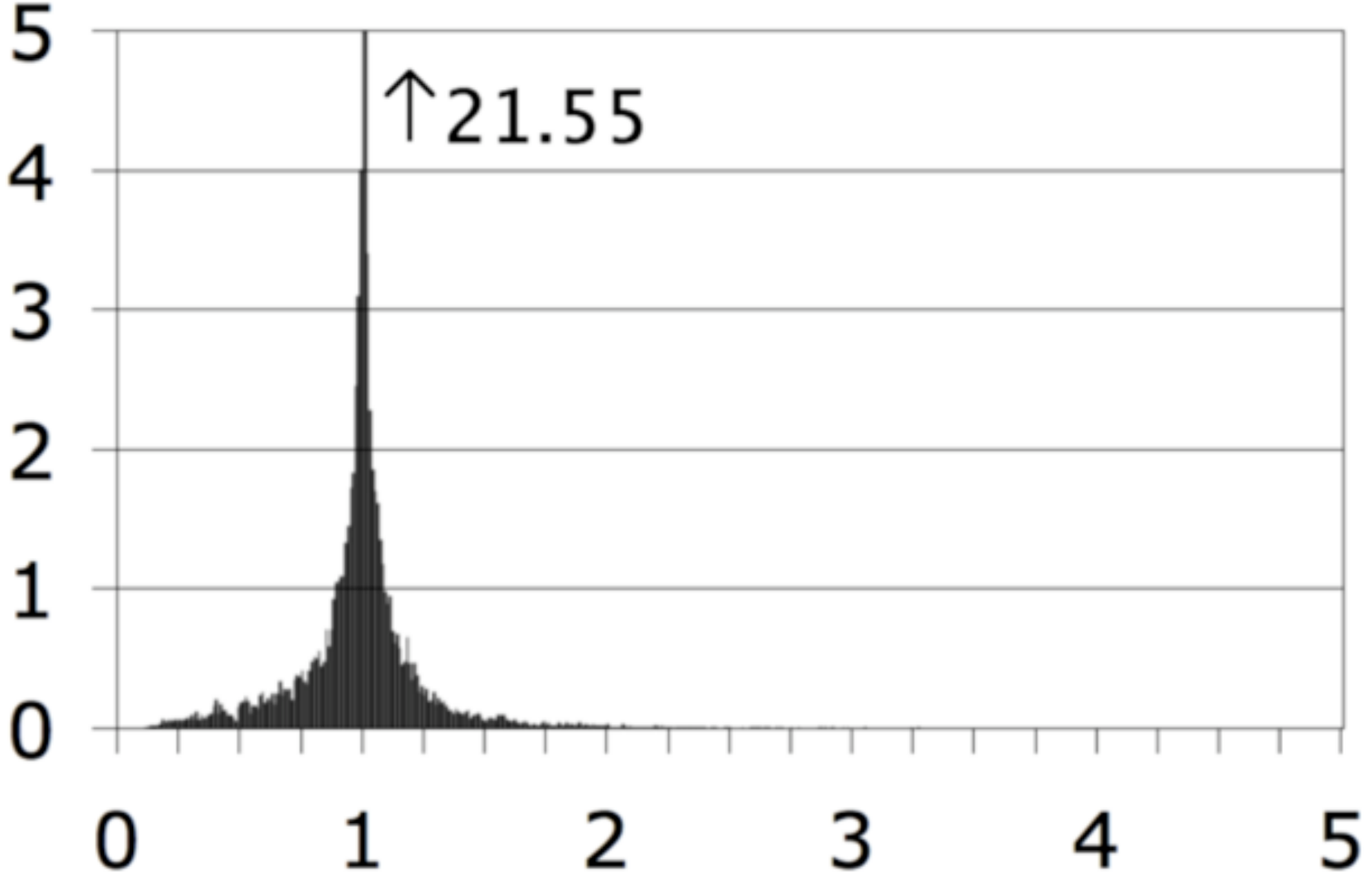}} &
\scalebox{0.17}{\includegraphics*[viewport = 0in 2.8in 8.5in 8.2in]{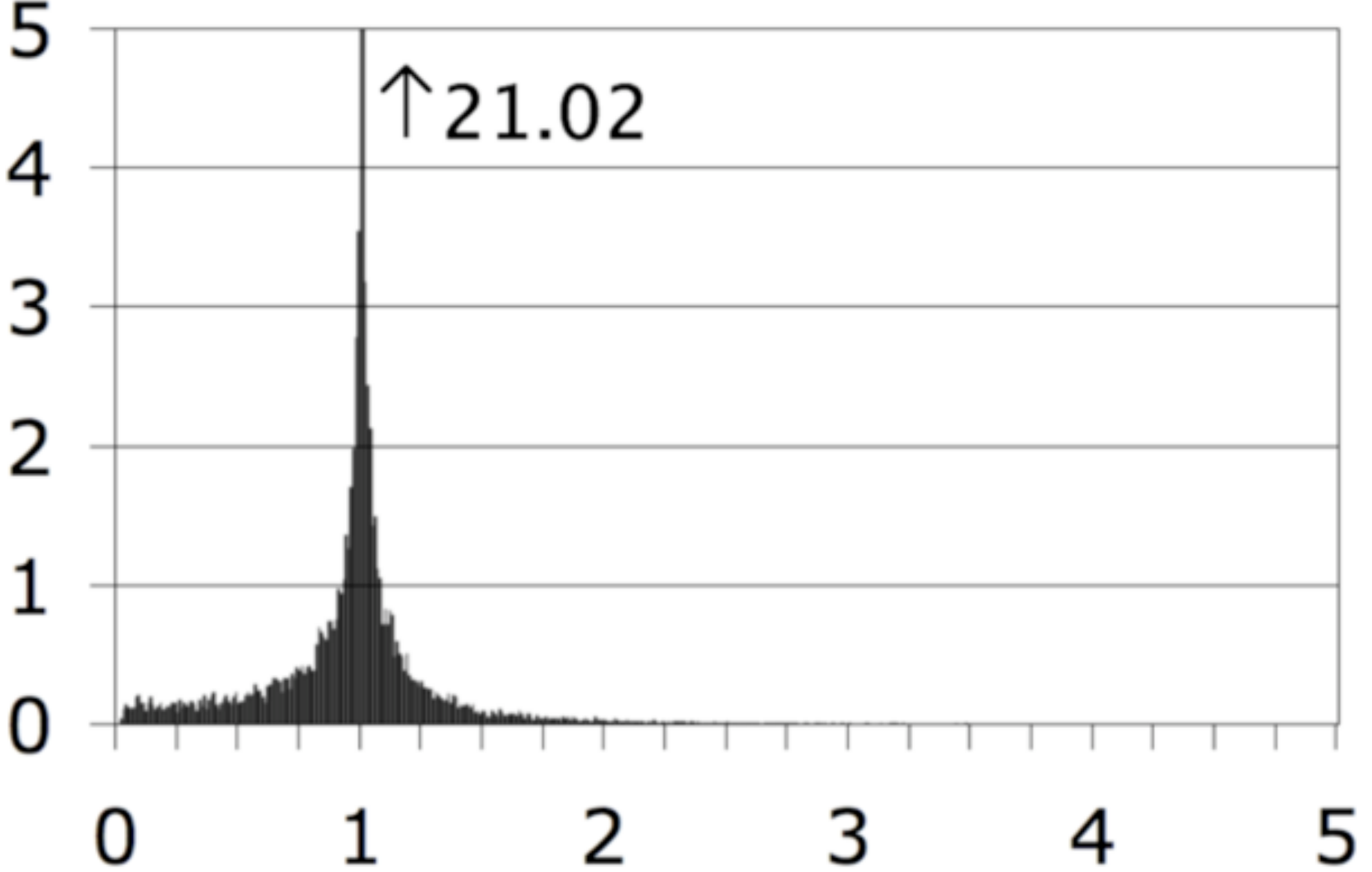}} &
\scalebox{0.17}{\includegraphics*[viewport = 0in 2.8in 8.5in 8.2in]{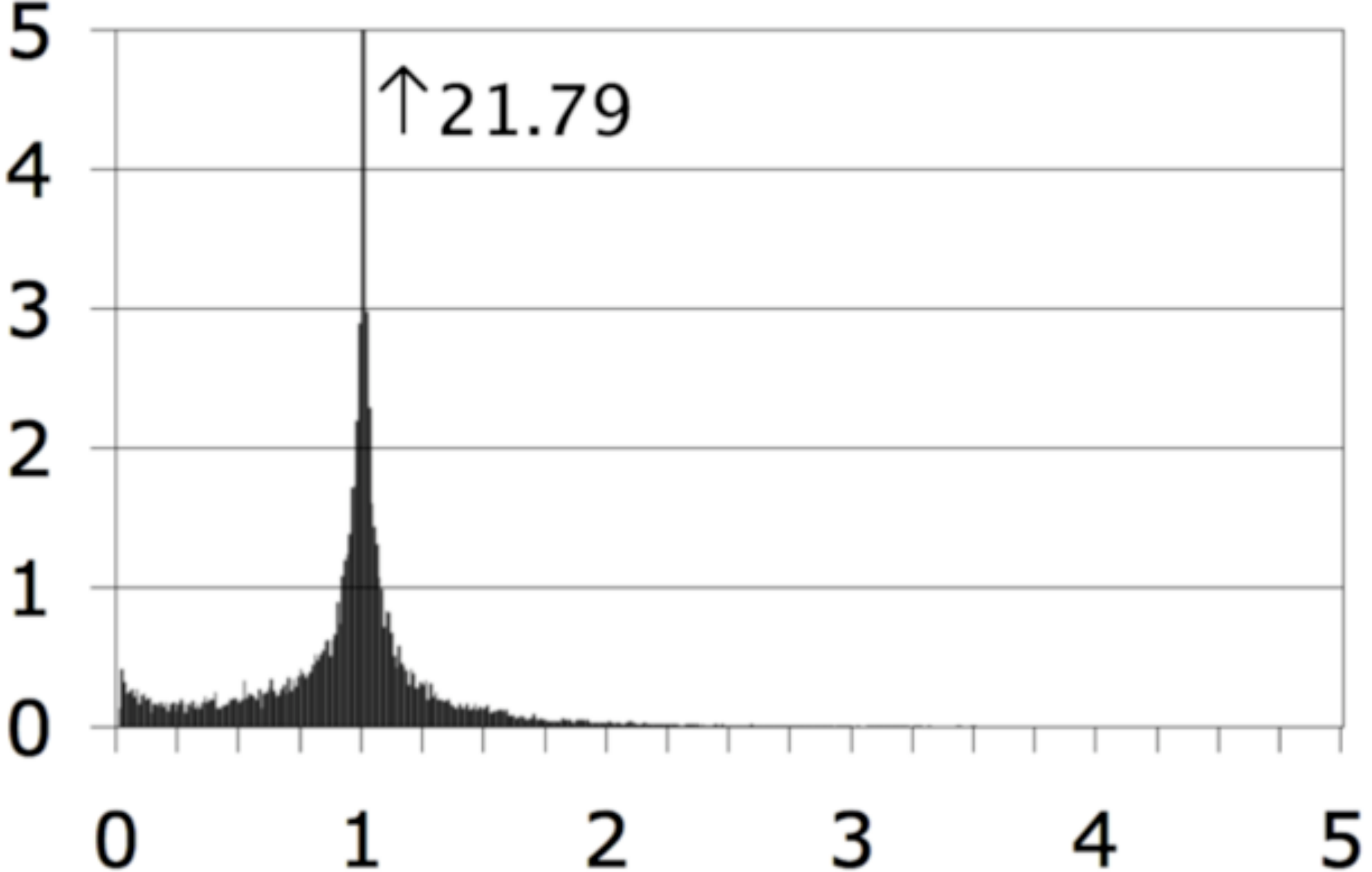}} \\
$\eta = 1$, row 5 & $\eta = 1$, row 20 & $\eta = 1$, row 50 \\
\scalebox{0.17}{\includegraphics*[viewport = 0in 2.8in 8.5in 8.2in]{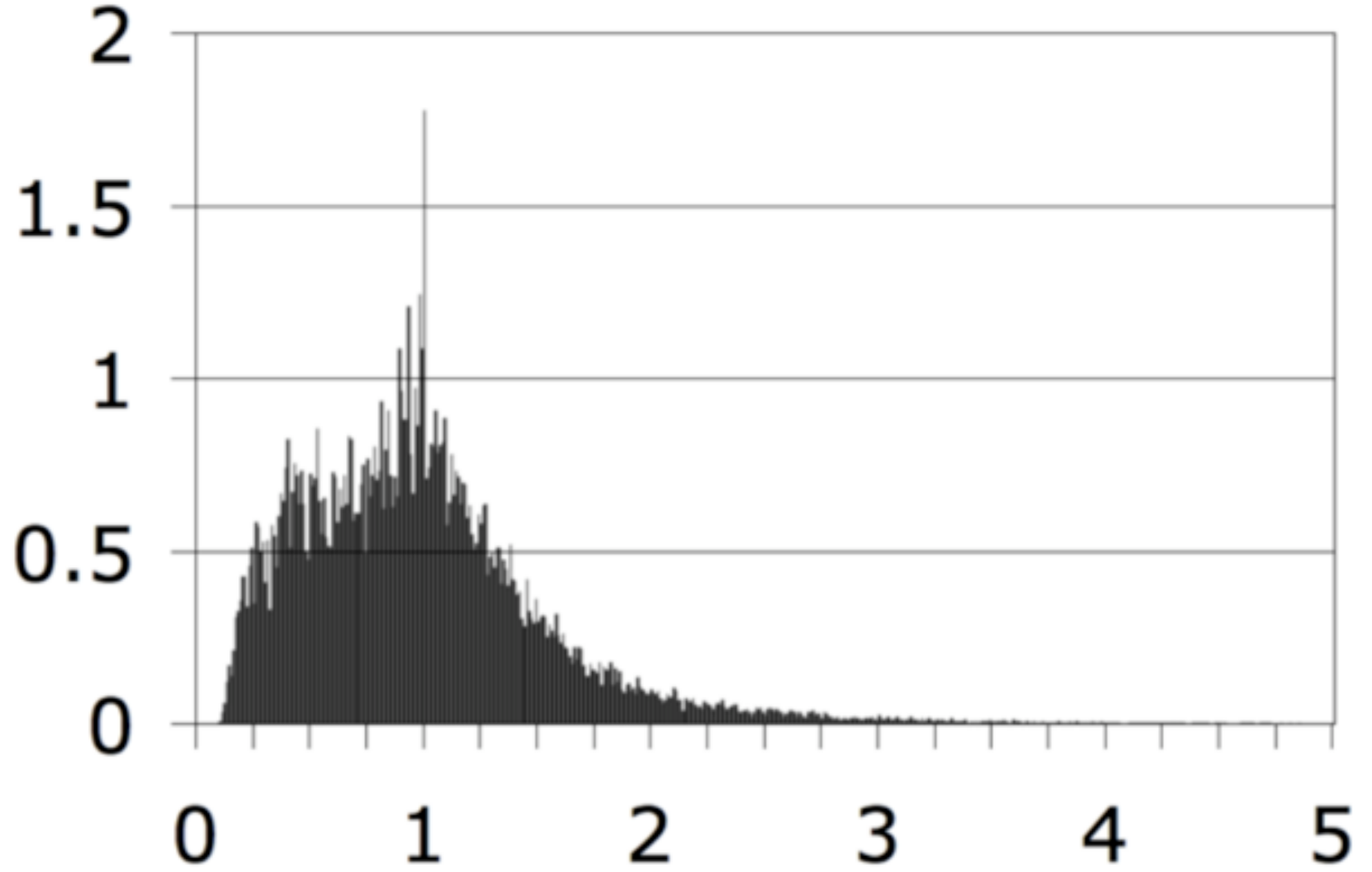}} &
\scalebox{0.17}{\includegraphics*[viewport = 0in 2.8in 8.5in 8.2in]{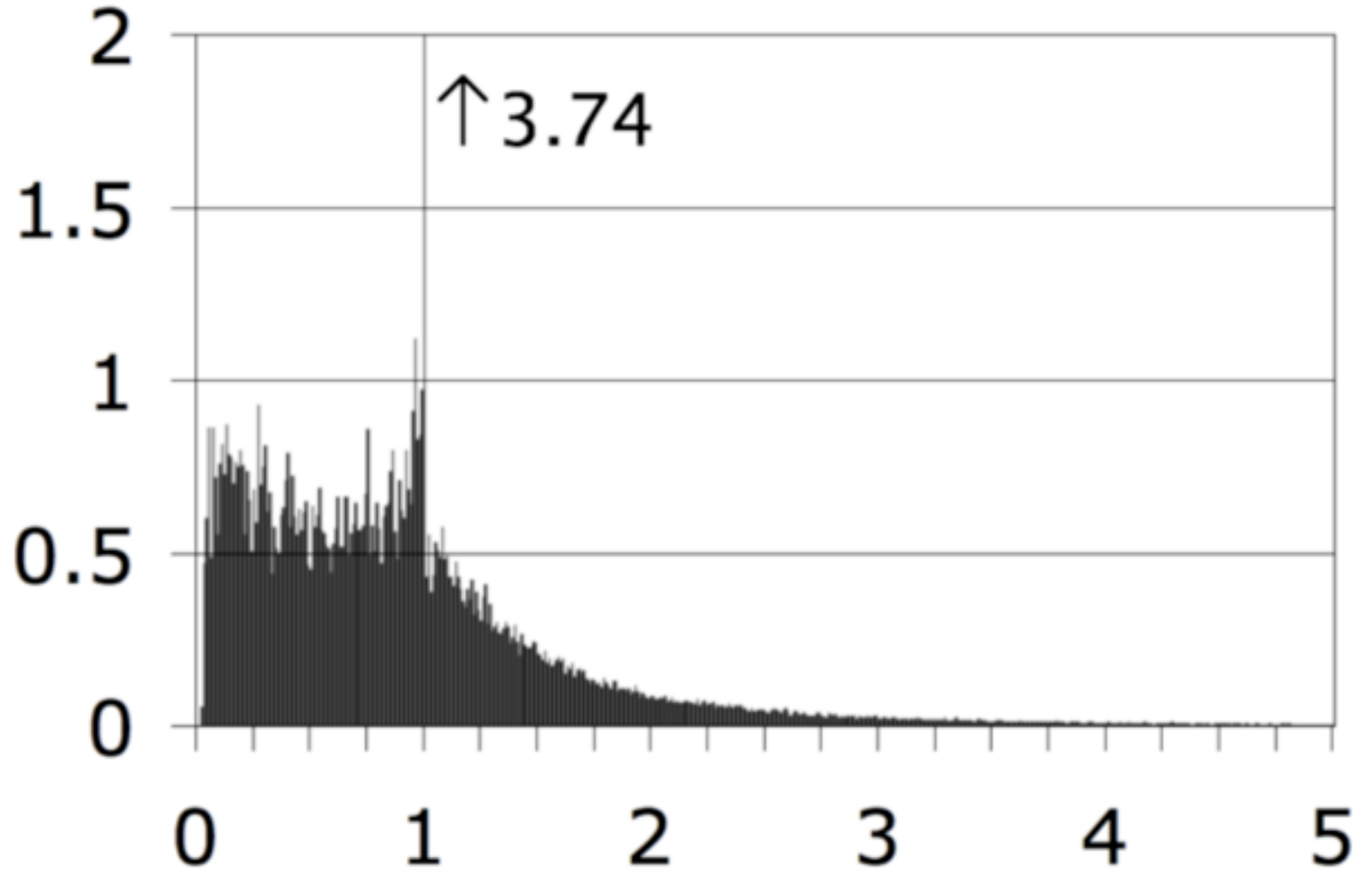}} &
\scalebox{0.17}{\includegraphics*[viewport = 0in 2.8in 8.5in 8.2in]{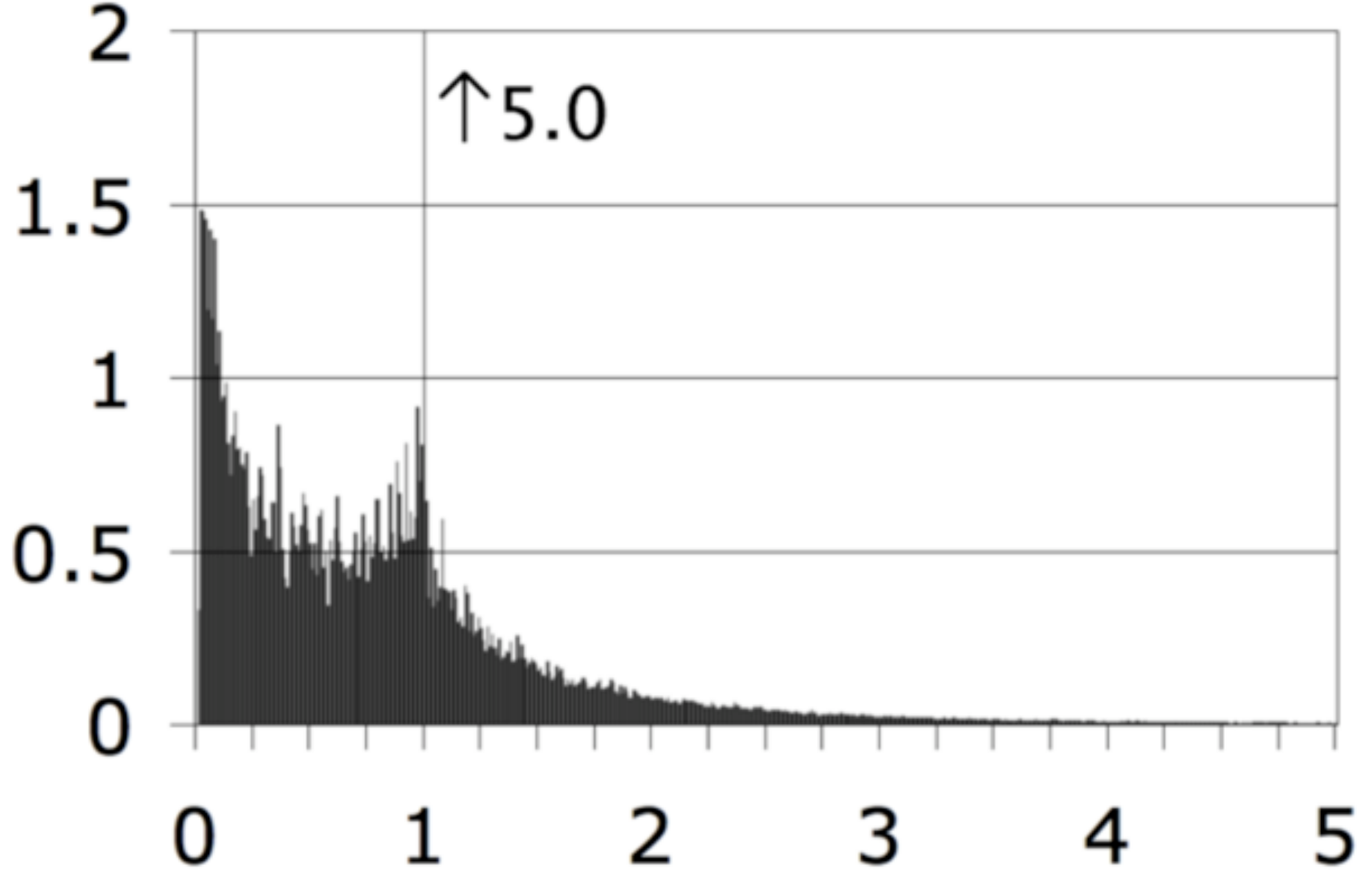}} \\
$\eta = 10$, row 5 & $\eta = 10$, row 20 & $\eta = 10$, row 50 \\
\scalebox{0.17}{\includegraphics*[viewport = 0in 2.8in 8.5in 8.2in]{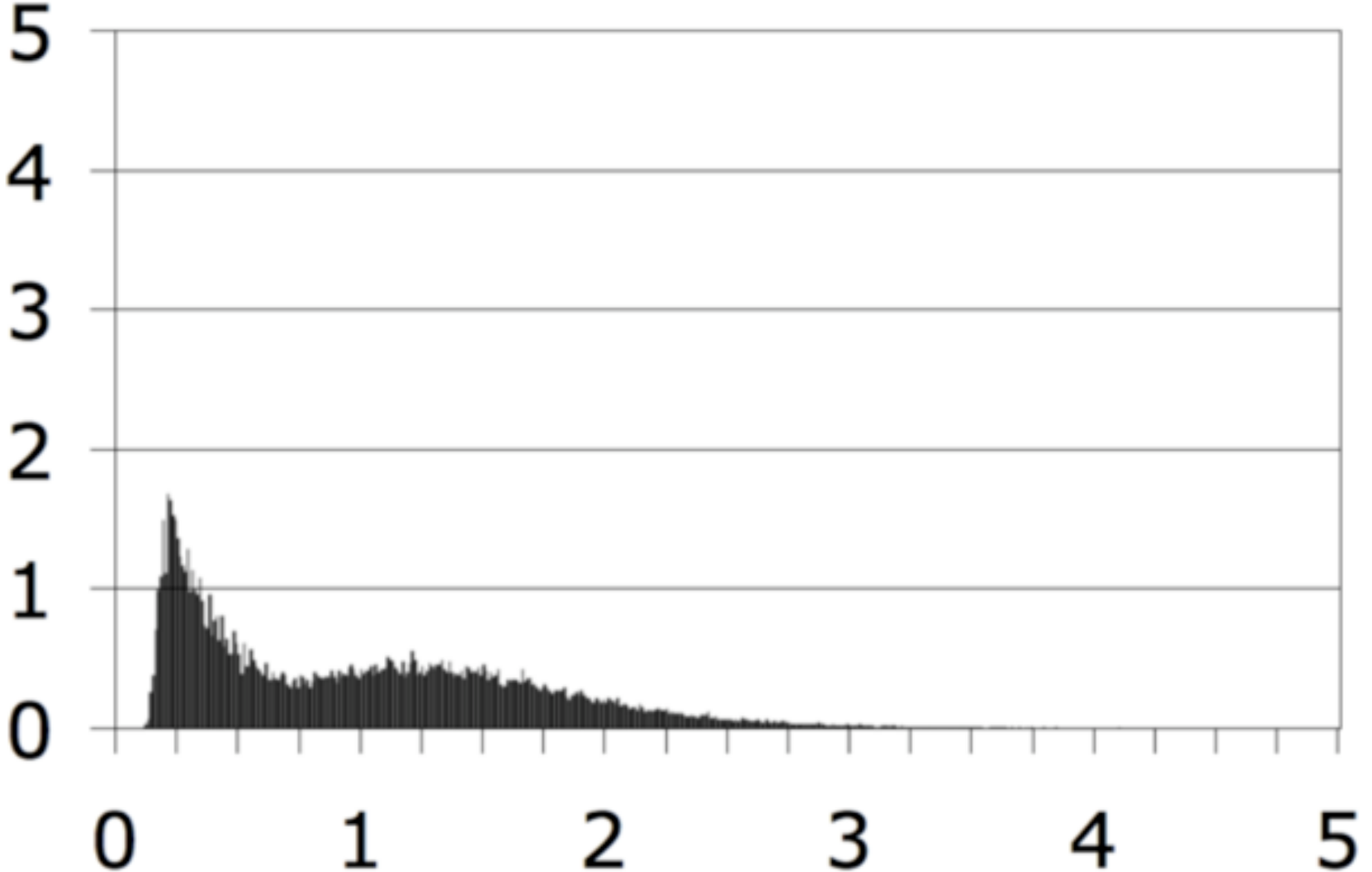}} &
\scalebox{0.17}{\includegraphics*[viewport = 0in 2.8in 8.5in 8.2in]{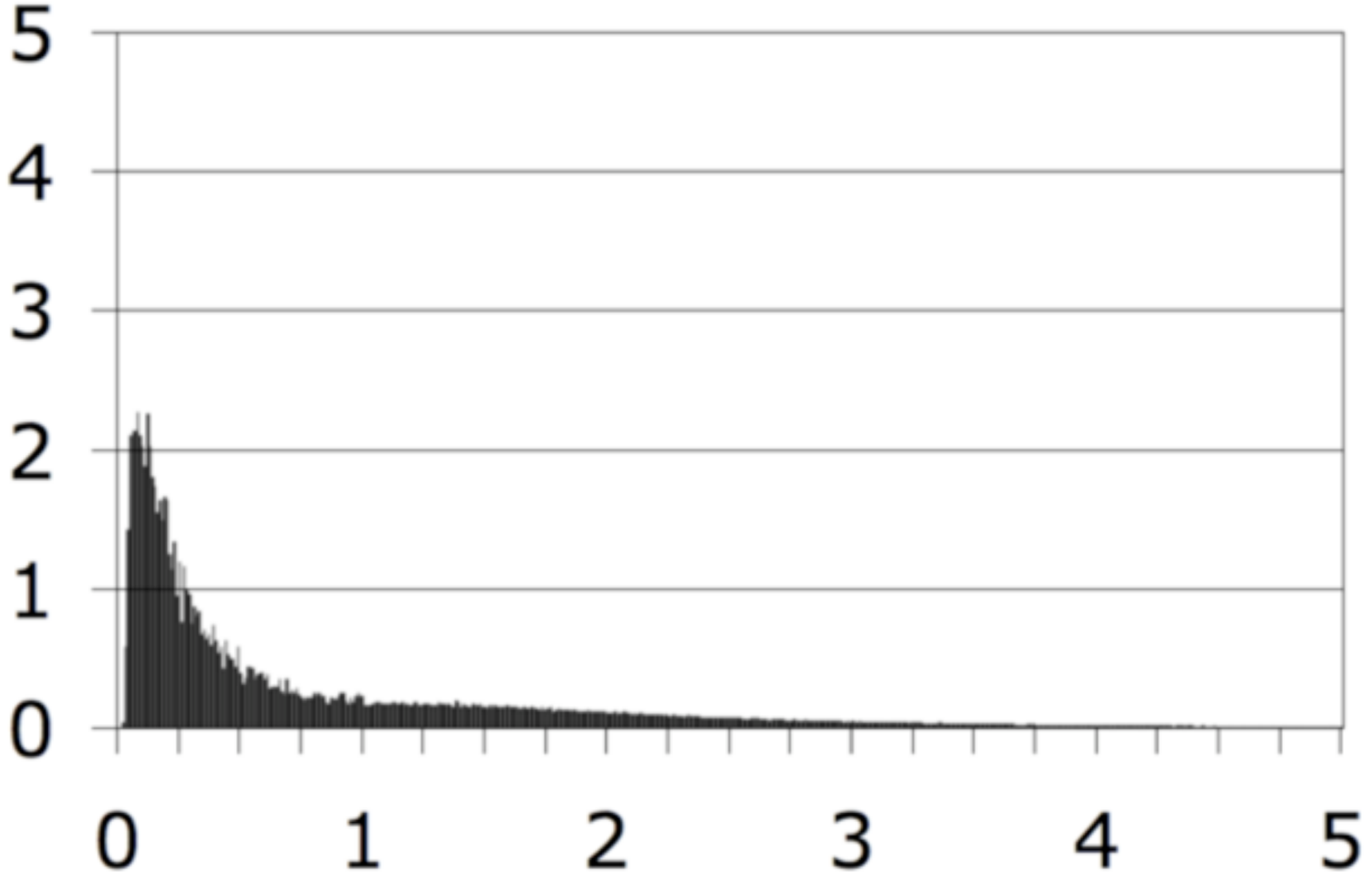}} &
\scalebox{0.17}{\includegraphics*[viewport = 0in 2.8in 8.5in 8.2in]{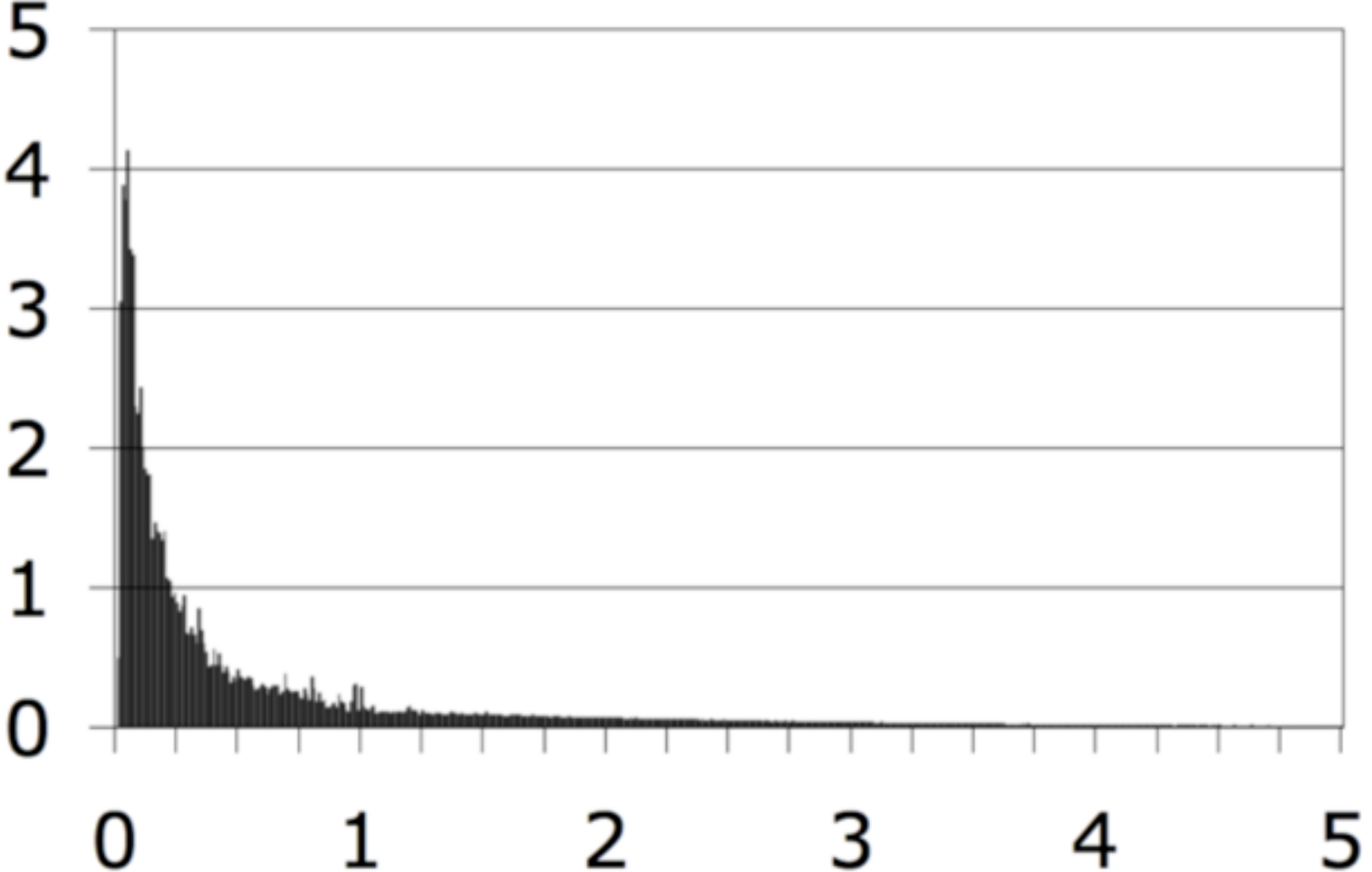}} \\
\end{tabular}
\end{center}
\caption{Histograms approximating the measure ${\mathbb E}(\mu_{v,n})$ for $n$ large.}
\label{floodfig}
\end{figure*}

\begin{table}
\centering
\begin{tabular}{| c | c c c|}
\hline
$\eta$ & Row 5 & Row 20 & Row 50 \\
\hline \hline
.1 & .00035 & .00205 & .00308 \\ 
1 & .00136 & .01235 & .01789 \\
10 & .00024 & .02356 & .03711 \\
\hline
\end{tabular}
\caption{Expected fraction of time during which a large flood (ratio at least 5) occurs.}
\label{largefloodtable}
\end{table}

We now change focus to the parent.  Make the following definition.

\begin{definition}
For any $n \geq 2$, define the \textbf{right catastrophe ratio} 

\[ C_v^R(n) = \frac{I_v(n)}{A_v^R(n)} \textrm{ whenever } D_v^L(n) = 0 \]  
Here, $A_v^R(n) = \frac{T_v^R(n-1)}{N_v^R(n-1)}$, where $N_v^R(n)$ is the number of $i \leq n$ such that $D_v^R(i) = 0$.  For $c \geq 1$, we say that a \textbf{right catastrophe} of order $c$ occurs at time $n$ if $C_v^R(n) \geq c$.  Make similar definitions for left catastrophe ratio and left catastrophe.
\end{definition}

\begin{remark}
From a similar argument to that used in Proposition \ref{floodprop}, we have

\[ \liminf_{n \to \infty} C_v^R(n) < \limsup_{n \to \infty} C_v^R(n) \]
\end{remark}

We investigate the relationship between floods and catastrophes.  Figure \ref{cattofloodfig} shows the expected fraction of right catastrophes from the left parent which result in a flood of at least the same magnitude.  The simulation was performed with a network of width 1000, depth 50, $\eta$ equal to either .1, 1, or 10, and for a duration of $10^5$ seconds.  The calculation of fractions was only made between 9000s and 10000s and we only consider catastrophes with ratio at least 1.  It is clear from the figure that as the row increases, this expectation decreases.  As $\eta$ decreases for a fixed row, the expectation also decreases.  As $n \to \infty$, 

\[ A_v^R(n) = \frac{(n-1)T_v^R(n-1)}{(n-1)N_v^R(n-1)} \to \lim_{n \to \infty} A_v(n) \]
by Remark \ref{rightavgrmk} and Lemma \ref{freqlemma}.  Therefore, we may use Remark \ref{noatomsavgloadrmk} to show that almost surely for large $n$ a right catastrophe of order $c$ occurs for the node $v$ at time $n$ if $v$ has a flood of order $c$ at time $n-1$.  In other words, whenever a node receives a flood of order $c$ at a large time, it has either a left or right catastrophe of the same order at the next second.  

Now we may interpret the probability that a node has a flood given that its left parent has a right catastrophe as the probability that a parent's catastrophe incites a catastrophe in the child.  This would be a step of a possible catastrophe \textit{cascade}.  The simulation results indicate that cascades become less present at lower levels (on average) but that they should never cease to exist.  Two questions naturally arise.  Given that a node has a right catastrophe of order $c$, how far does its catastrophe cascade travel?  At each step of the cascade, the relevant (right or left) catastrophe ratio will generally increase.  Indeed, a child node may even receive a catastrophe from both parents.  How large does this ratio become in a typical cascade?

\begin{figure*}
\begin{center}
\scalebox{0.55}[0.45]{\includegraphics*[viewport = 0in 2.75in 8.5in 8.25in]{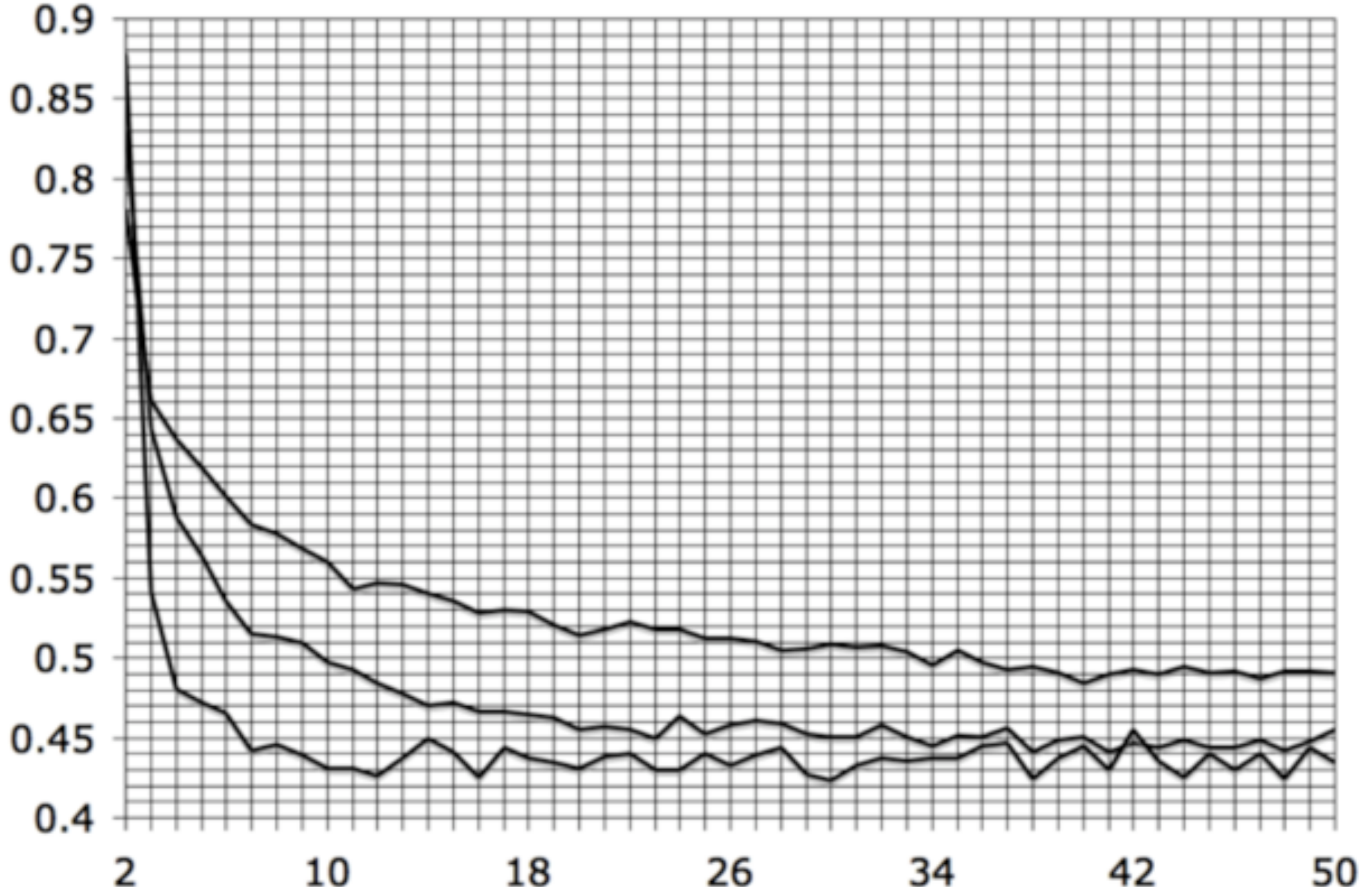}}
\end{center}
\caption{Expected fraction of right catastrophes from left parent which result in a flood of at least the same magnitude.  The values of eta are 10 (top), 1 (middle), and .1 (bottom).  The data are plotted by row.}
\label{cattofloodfig}
\end{figure*}

\section{Conclusion}
We have shown that the erosion model exhibits many properties of rill erosion.  Each node chooses a random initial direction (right or left) in which to send sediment and further such choices become biased at a rate largely determined by the parameter eta.  This is similar to the method by which rills are cut into a hillslope.  As more water and sediment flows through a rill, a channel is cut deeper, giving reinforcement to the path, making it more likely to carry sediment in the future.  Though the dynamics manifests itself through reinforcement, no fixed node can become fully biased (i.e. have a de Finetti measure equal to a sum of two delta masses).  That is, since each node has a non-trivial asymptotic switching rate, sediment flow emerging from it will take both a left and right path a positive fraction of time.  This rate of switching appears to decrease as we move further down the hill.

There are a number of questions which deserve careful analysis.  Do the measures $\theta_k$ have a limit?  If so, one would expect the limit to be the de Finetti measure associated with the "infinity process".  To define this process, we start with a lattice of nodes which extends infinitely far in both positive and negative $y$ directions.  Since the behavior of a node $v$ at time $n$ in the present model depends only on the nodes in the $n-1$ levels above it we may consider the input to the node $v$ at time $n$ in the infinity model to be a function of the output of this finite number of ancestors.  In the same way we have analyzed in this paper, it is possible to show that a de Finetti measure $\theta_{\infty}$ for this process exists and that 

\begin{equation}
\theta_{\infty} = \lim_{n \to \infty} \theta_n(n),
\end{equation}
where the term inside the limit is the measure given by

\[ \theta_n(n) ([a,b]) = {\mathbb P}( P_v^L(n) \in [a,b]) \textrm{ for depth}(v) = n \]
Does the measure $\theta_{\infty}$ have atoms for some values of $\eta$?  If so, is there a critical $\eta_*$ so that for $0 < \eta < \eta_*$, $\theta_{\infty}$ has atoms?  If the limit of $\theta_k$ (assuming it exists) is not the same as $\theta_{\infty}$, does this limit have atoms and is there a critical $\eta$ associated with it?

\vspace{.2in}
\textbf{Acknowledgments}.  We thank Everett Springer for careful reading and comments.  We also thank Charles Newman, Daniel Stein, Emmanuel Schertzer, and Itai Benjamini for very helpful discussions.  Research at NCAR was funded by the National Science Foundation.


\begin{thebibliography}{2}
\bibitem[1]{athreya} Athreya, K.: On a characteristic property of P\'{o}lya's urn.  \textit{Stud. Sci. Math. Hung.} \textbf{4}, 31--35 (1969) 
\bibitem[2]{arratia} Arratia, R.: Coalescing Brownian motions and the voter model on ${\mathbb Z}$, Unpublished partial manuscript (1981), available from rarratia@math.usc.edu.
\bibitem[3]{diaconis} Diaconis, P.: Recent progress on de Finetti's notion of exchangeability.  In J. Bernardo, M. de Groot, D. Lindley, and A. Smith, editors, \textit{Bayesian Statistics}, pp. 111--125. Oxford University Press, Oxford (1988)
\bibitem[4]{polya} Eggenberger, F., P\'{o}lya, G.: \"{U}ber die Statistik Verketter Vorg\"{a}nge. \textit{Z. Angew. Math. Mech.} \textbf{3}, 279--289 (1923)
\bibitem[5]{feller} Feller, W. \textit{Introduction to Probability Theory and Its Applications,} Vol. 2, Wiley, New York.
\bibitem[6]{FINR} Fontes, L., Isopi, M., Newman, C.M., Ravishankar, K.: Coarsening, nucleation and the marked Brownian web.  \textit{Ann. Inst. H. Poincare.} \textbf{42} 37--60 (2006)
\bibitem[7]{FNRS} Fontes, L., Newman, C., Ravishankar, K., Schertzer, E.: The Dynamical Discrete Web. \textit{Arxiv: math.PR/07042706} (2007)
\bibitem[8]{freedman} Freedman, D.: Bernard Friedman's urn.  \textit{Ann. Math. Statist.}, \textbf{36} 956--970 (1965)
\bibitem[9]{friedman} Friedman, B.: A simple urn model. \textit{Commun. Pure Appl. Math.} \textbf{2}, 59--70 (1949)
\bibitem[10]{glock} Glock, W.S.: The development of drainage system -- a synoptic view. \textit{Geograph. Rev.}, \textbf{21}, 475 (1931)
\bibitem[11]{howittwarren} Howitt, C., Warren, J.: Dynamics for the Brownian web and the erosion flow.  \textit{Arxiv: math.PR/0702542}.
\bibitem[12]{kingman} Kingman, J.: Uses of exchangeability. \textit{Annals of Probability} \textbf{2} 183--197 (1978)
\bibitem[13]{CAP} May, C., Paganoni, A., Secchi, P.: On a two-color generalized P\'{o}lya urn.  \textit{International Journal of Statistics} \textbf{63}, 115--134 (2005)
\bibitem[14]{pemantle} Pemantle, R.: A time-dependent version of P\'{o}lya's urn.  \textit{Journal of Theoretical Probability} \textbf{3}, 627--637 (1990)
\bibitem[15]{pemantle2} Pemantle, R.: A survey of random processes with reinforcement.  \textit{Probability Surveys} \textbf{4} 1--79 (2007)
\bibitem[16]{quinton} Quinton, John N.: Erosion and sediment transport, in \textit{Environmental modelling : finding simplicity in complexity}, ed. By John Wainwright and Mark Mulligan, Wiley, Hoboken, NJ, 2004.
\bibitem[17]{RRI} Rodr\'{i}guez-Iturbe, I., Rinaldo, A.: \textit{Fractal River Basins: Chance and Self-Organization}, Cambridge University Press, Cambridge, UK, 1997.
\bibitem[18]{toy} Toy, T.J., Foster, G.R., Renard, K.G.: \textit{Soil Erosion: Processes, Prediction, Measurement, and Control}, John Wiley and Sons, N.Y., 2002.
\bibitem[19]{WBR} Willgoose, G. R., Bras, R. L., and Rodriguez-Iturbe, I.: 1991a. A Physically Based Coupled Network Growth and Hillslope Evolution Model: 1 Theory. WRR, 27(7), 1671--1684.
\end{thebibliography}
\end{document}